\documentclass[a4paper, 10pt]{article}
\usepackage[latin1]{inputenc}
\usepackage{amsthm}
\usepackage{amsmath,amsfonts}
\usepackage{color}
\usepackage{wrapfig}
\usepackage{subfigure}
\usepackage{graphicx}
\usepackage{bm}
\usepackage{bbm}
\usepackage{indentfirst}
\usepackage[margin=0.8in]{geometry}
\newtheorem{theorem}{Theorem}[section]
\newtheorem{proposition}{Proposition}[section]
\newtheorem{definition}{Definition}[section]
\newtheorem{lemma}{Lemma}[section]

\newtheorem{remark}{Remark}[section]

\numberwithin{equation}{section}

\def\e{\varepsilon}

\def\O{\Omega}

\def\ds{\displaystyle}
\def\set{/\kern-.51em o}
\def\eq{\mathop{\vrule height2,6pt depth-2,3pt
         width -1pt\kern 0pt =}}

\def\R{{\mathbb{R}}}
\def\Z{{\mathbb{Z}}}
\def\V{{\mathbb{V}}}
\def\K{{\mathbb{K}}}
\def\H{{\mathbb{H}}}

\begin{document}
\begin{center}
{\bf \large Homogenization via unfolding in periodic layer with contact}
\end{center} 
  \centerline{Georges Griso, Anastasia Migunova, Julia Orlik}

\begin{abstract}
In this work we consider the elasticity problem for two domains separated by a heterogeneous layer. The layer has an $\e-$periodic structure, $\e\ll1$, including a multiple micro-contact between the structural components. The components are surrounded by cracks and can have rigid displacements. The contacts are described by the Signorini and Tresca-friction conditions. In order to obtain preliminary estimates modification of the Korn inequality for the $\e-$dependent periodic layer is performed.

An asymptotic analysis with respect to $\e\to 0$ is provided and the limit problem is obtained, which consists of the elasticity problem together with the transmission condition across the interface. The periodic unfolding method is used to study the limit behavior.
\end{abstract}

\section{Introduction}
Contact problems for the domains with highly-oscillating boundaries were considered in different works (see e.g. \cite{kikuchi, mik, yosif}) as well as problems on the domains including thin layer \cite{cior, nrj, geym}. The elasticity problem for a heterogeneous domain with a Tresca--type friction condition on a microstructure (involving inclusions and cracks) was considered in \cite{cior3} and Korn's inequality for disconnected inclusions was obtained. 

In this paper we are concerned with an elasticity problem in a domain containing a thin heterogeneous layer of the small thickness $\e$. We consider the case in which the stiffness of the layer is also of the order $\e$. The contact between structural components in the layer is described by the Tresca--friction contact conditions. Our aim is to study the behavior of the solutions of the microscopic equations when $\e$ tends to zero.

For the derivation of the limit problem we use the periodic unfolding method which was first introduced in \cite{cior2}, later developed in \cite{CiDaGr} and was used for different types of problems, particularly, contact problem in \cite{cior3} and problems for the thin layers in \cite{cior}. 

However, additional difficulties arise in the proving uniform boundedness for the minimizing sequence of the solutions. The idea consists in controlling the norm by the trace on the boundary of the domain and, therefore, by the norm on the outside domain. Working this way we first obtain estimate for the Dirichlet domain, through which an estimate of the trace is obtained and, therefore, norm on the structural components of the layer. Two Korn's inequalities are introduced: for the inclusions placed in the heterogeneous periodic layer (based on the results from \cite{cior3}) and for the connected part of the layer. The main result of the study is an asymptotic model for the layer between elastic blocks.

The paper is organized as follows. Section 2 gives the geometric setting for the $\e$-periodic problem, including the unit cell. In Section 3 we give inequalities related to the unfolding operator on the interface surfaces, then establish a uniform Korn inequality for the perforated matrix domain. Then, two unilateral Korn inequlities are proved with their applications to the oscillating inclusions and the matrix of the layer. Section 4 deals with the convergence result. In Section 5 the problem for fixed $\e$ is introduced. At last, in Section 6 the limit problem is obtained and the case of the linearized contact conditions is considered.

\subsection{Notations}
\begin{itemize}
 \item Let $\mathcal O$ be a bounded domain in $\mathbb R^3$ with a Lipschitz boundary. For any $v \in H^1 (\mathcal O; \mathbb R^3)$,  the normal component of a vector field $v$ on the boundary of $\mathcal O$ is denoted $v_{\nu} = v_{|\partial O} \cdot \nu$, while the tangential component $v_{|\partial O} - v_{\nu} \nu$ is denoted $v_{\tau}$ ($\nu$ is the outward unit normal vector to the boundary). 

 \item Let $S^0$ be a closed set in $\mathbb R^3$: a finite union  of disjoint orientable surfaces of class ${\cal C}^1$. Then for every piece of surface, we choose a continuous field of unit normal vector denoted $\nu$.
 
 \item The strain tensor of a vector field $v$ is denoted by $e (v)$,
$$ e^{ij} (v)= \frac{1}{2} \Big( \frac{\partial v_i}{\partial x_j} + \frac{\partial v_j}{\partial x_i} \Big) \quad (i, j) \in \{ 1, 2, 3 \}^2.$$
The kernel of $e$ in a connected domain is the finite dimensional space of rigid motions denoted by $\mathcal R$.
\end{itemize}

\section{Geometric statement of the problem}

In the Euclidean space $\mathbb R^2$ consider a connected domain $\omega$ with Lipschitz boundary and let $L > 0$ be a fixed real number. Define:

%{\clb I think that the notations $\Omega^{b}$, $\Omega^{a}$ and later $u^{a}$, $u^{b}$ for the limit fields are too confusing since we use the positive part of  functions in many equations (see $(3.6)$, $(3.21)$, $(5.16)$). I propose to replace $\Omega^{b}$, $\Omega^{a}$ by $\Omega^b$, $\Omega^a$ (for below and above). Of course we also need to change $u^{a}$ and $u^{b}$, $\ldots$).
%
%
%
%}
\begin{gather*}
 \begin{array}{ll}
  \Omega^{b} &= \omega \times (-L, 0),\\
  \Omega^{a} &= \omega \times (0, L),\\
  \Sigma &= \omega \times \{0\},\\
  \Omega &= \Omega^{a} \cup \Omega^{b} \cup \Sigma=\omega\times (-L,L).
 \end{array}
\end{gather*}
To describe a structure with a layer introduce the notations:
\begin{equation}
\label{doms}
\begin{array}{ll}
 \Omega_{\e}^{a} &= \omega \times (\e, L),\\
 \Omega_{\e}^M &= \omega \times (0, \e),\\
 S_{\e}^{a} &= \omega \times \{ \e \},
\end{array}
\end{equation}
%and let $\Omega_{\e} = \Omega_{\e}^{a} \cup \Omega_{\e}^M \cup \Omega^{b} \cup \Sigma \cup S_{\e}^{a}$ be the physical reference configuration of the assemblage (see Figure \ref{gfig}). 
Here $\e$ is a small parameter corresponding to the thickness of the layer.

The assemblage is fixed on $\Gamma$, which is a non-empty part of $\partial \Omega^{b}$ ($\Gamma$ is a set where the Dirichlet condition will be prescribed). Furthermore, we assume that the external boundary of the layer $\partial \Sigma \times [0, \e]$ is traction free.

%\begin{figure}[h!]
%\center
%\includegraphics[width = 0.4\textwidth]{pic1.jpg}
%\caption{Domain $\Omega_{\e}$}
%\label{gfig}
%\end{figure}

\begin{wrapfigure}{r}{0.38\textwidth}
\center
\includegraphics[width = 0.27\textwidth]{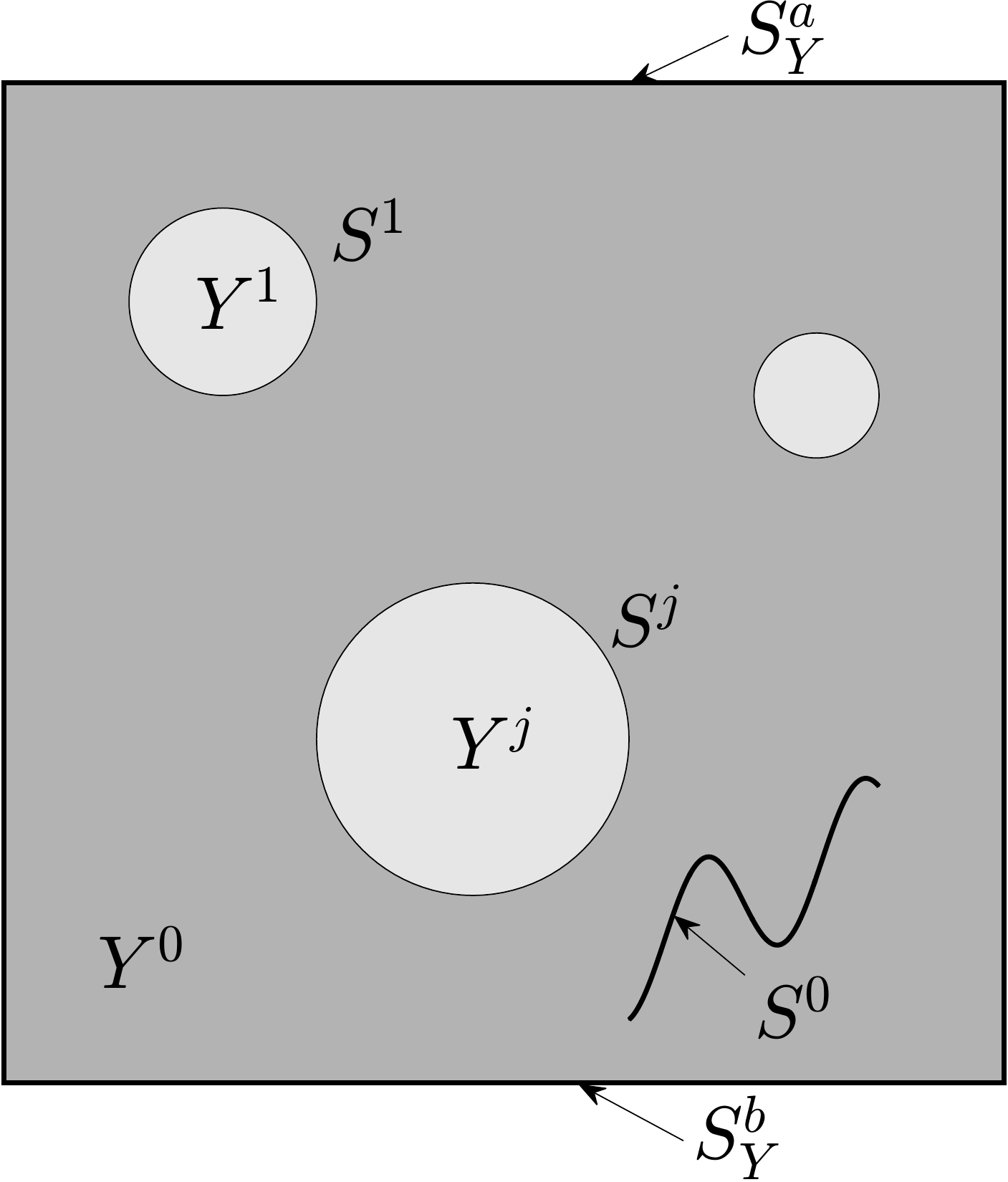}
\caption{The unit cell $Y$}
\label{fg}
\end{wrapfigure}

The layer $\Omega_{\e}^M$ has periodic in-plane structure. The unit cell is denoted by $Y$
\begin{equation*}
Y = \big( 0, 1 \big)^3\subset \R^3, \qquad Y_{\e} = \e Y.
\end{equation*}
Additionally,
\begin{equation*}
 S^{b}_Y = \left\{ y \in Y : y_3 = 0\right\},\qquad  S^{a}_Y = \left\{ y \in Y : y_3 =1 \right\}
\end{equation*}
are the lower and upper boundaries of $Y$.

There are two kinds of cracks, the first ones $S^1, \ldots, S^m$ (the ``closed cracks'') are the closed boundaries of open Lipschitzian sets $Y^j, j \in \{ 1, \ldots, m\}$. We assume that every $S^j, j \in \{ 1, \ldots, m \}$, has only one connected component and $ \bigcup_{j = 1}^m \overline{Y^j } \subset Y$. The other cracks (the ``open cracks''), which union is denoted by $S^0$, are the finite union of closed Lipschitz surfaces included in  $Y \setminus \bigcup_{j = 1}^m \overline{Y^j }$ (see Figure \ref{fg}). 
\medskip

%{\clb You must change the figure 1}

We set
$$Y^0=Y\setminus \Big(\bigcup_{j=1}^{m}\overline{Y^{j}}\cup S^0\Big)$$

\noindent and we assume that there exists $t_0>0$ such that 
$$\forall x^{'}\in S^0,\qquad [x^{'}-t_0\nu(x^{'}), x^{'}+t_0\nu(x^{'})]\setminus \{x^{'}\}\subset Y^0.$$

Since the set of cracks $\ds \bigcup_{j=0}^m S^j$ is a closed subset strictly included in $Y$, there exists $\eta>0$ such that
$$\forall x\in \bigcup_{j=0}^m S^j, \qquad \hbox{dist}(x, \partial Y)\geq \eta.$$
Denote
$$Y'=(0,1)^2.$$
Recall that in the periodic setting almost every point  $z\in \mathbb R^3$ (resp. $z'\in \R^2$) can be written as 
$$
\begin{aligned}
z&= \left[z\right]_Y + \left\{z\right\}_Y,\qquad \quad \left[z\right]_Y \in \Z^3,\quad \left\{z\right\}_Y\in Y,\\
\text{( resp. } z'&= \left[z'\right]_{Y'}+ \left\{z'\right\}_{Y'},\qquad \left[z'\right]_{Y'} \in \Z^2,\quad \left\{z'\right\}_{Y'}\in Y'\text{)}.
\end{aligned}
$$

Denote by
\begin{itemize}
 \item $\ds \Xi_{\e}=\big\{\xi\in \Z^2\;|\; \e (\xi + \e Y')\subset \omega\big\}$, $\ds \Xi_{\e}^M=\Xi_\e\times\{0\}$,
 \item $\widehat\omega_{\e}= \text{interior }\ds \Big( \bigcup_{\xi \in \Xi_{\e}} \e (\xi + \overline{Y'}) \Big)$,  $\widehat\O_{\e}^M = \text{interior }\ds \Big( \bigcup_{\xi \in \Xi^M_{\e}} \e (\xi +\overline{Y}) \Big)=~\widehat{\omega}_\e\times(0,\e)$,
 \item $\Lambda_{\e}=\omega \setminus \overline{\widehat\omega_{\e}}$, $\Lambda_{\e}^M=\O_{\e}^M \setminus \overline{\widehat\O_{\e}^M}=\Lambda_{\e}\times (0,\e)$.
\end{itemize}
The open subset of $\O_{\e}^M$ contains the parts of cells intersecting the lateral boundary $\partial \omega\times(0, \e)$.

For $j=1,\ldots, m$ introduce the set  
$$\O_{\e}^{j} = \Big\{x \in \,\widehat\O_{\e}^M \;\;  |  \;\;\e\left\{\frac{x}{\e}\right\}_{Y} \in Y^{j} \Big\}.$$
The boundary  $\partial \O_{\e}^{j}$ is the set  of ``closed cracks'' associated with $S^{j}$, 
$$ \partial \O_{\e}^{j} \doteq S_{\e}^{j} = \Big\{x\in \,\widehat\O_{\e}^M\;\; | \;\;\e\left\{\frac{x}{ \e}\right\}_{Y} \in S^{j}=\partial Y^{j} \Big\} \;  .$$
For $j=0$ set 
$$S_{\e}^{0} = \Big\{x \in \,\widehat\O_{\e}^M\;\; | \;\;\e\left\{\frac{x}{ \e}\right\}_{Y} \in S^{0} \Big\}$$
and 
$$
\Omega_{\e}^0 \doteq \displaystyle \O_{\e}^M \setminus\Big( {\mathop{\bigcup}\limits _{j=1,\ldots, m}} \overline{\O_{\e}^{j}} \cup S_{\e}^{0}\Big),\quad \qquad \widehat{\Omega}_{\e}^0 \doteq\Omega_{\e}^0 \cap \widehat{\Omega}^M_\e.
$$
The union of all the cracks is denoted by $S_{\e}$ 
$$S_{\e} =  {\mathop{\bigcup}\limits _{j=0,1, \ldots, m}} S_{\e}^{j}.$$
We define the set $\Omega_{\e}$ and $\Omega_{\e}^*$ as follows
$$
\Omega_{\e} \doteq \Omega \setminus S_{\e}\qquad \quad \Omega^*_{\e} \doteq \Omega\setminus\Big( {\mathop{\bigcup}\limits _{j=1,\ldots, m}} \overline{\O_{\e}^{j}} \cup S_{\e}^{0}\Big).
$$
Note that from these definitions it is clear that there are no cracks in the part of the layer $\Lambda_{\e}^M$.
\medskip

For a function $v$ defined on $\O^*_{\e}$, for simplicity, we denote its restriction to $\O_{\e}^{j}$ by $v^{j}$ 
\begin{equation*}
\label{vj}
v^{j}  \doteq v_{|\O_{\e}^{j}} \quad \hbox {for }  j= 1,\ldots,m.
\end{equation*}
%
%We use the same notation $\mathcal T_{\e}^{bl}$ for the boundary-layer unfolding operator with all its properties from \cite{cior} with slight extension for the cracks $S_{\varepsilon}^j$.

In the following, for any bounded set ${\cal O}$ and  $\varphi\in L^{1}({\cal O})$, ${\cal M}_{{\cal O}}(\varphi)$ denotes the mean value of $\varphi$ over ${\cal O}$, i.e.
$${\cal M}_{\cal O}(\varphi) = \frac{1}{ |{\cal O}|}\int_{\cal O} \varphi\, dy.$$

\section{Some inequalities related to unfolding and the geometric domain}
\subsection{The boundary-layer unfolding operator $\mathcal T_{\varepsilon}$, $\mathcal T^{bl, j}_{\varepsilon}$}
Here we recall the definition and the properties of the boundary-layer unfolding operator (for more details see \cite{cior}). 
%Further we use the notation
%$$x' = (x_1, x_2).$$
  
\begin{definition}
 For $\varphi$ Lebesgue-measurable  on $\O_{\e}^M$ (resp. on $\Omega^{j}_\e,\;\; j=1,\ldots,m$), the unfolding operator $\mathcal T_{\e}$ is defined by
 \begin{equation*}
 \mathcal T_{\varepsilon}(\varphi)(x',y)=\left\{
 \begin{array}{ll}
 \ds \varphi \Big(\e\Big [\frac{(x', 0)}{\e}\Big]_Y + \e y\Big) & \text{ a.e. for } (x',y) \in \widehat \omega_{\e} \times Y,\quad \text{(resp. for a.e. }(x',y) \in \widehat \omega_{\e} \times Y^j\text{)}\\
 0 & \text{ a.e. for } (x',y) \in \Lambda_\varepsilon \times Y,\quad \text{(resp. for a.e. }(x',y) \in \Lambda_\varepsilon \times Y^j\text{)}.
 \end{array}
 \right.
 \end{equation*}
  For $\varphi \in L^p(S^{j}_\e)$, $j\in 1,\ldots, m$, Lebesgue-measurable  on $S^{j}_{\e}$, the unfolding operator $\mathcal T^{bl, j}_{\varepsilon}$ is defined by
 \begin{equation*}
 \mathcal T^{bl, j}_{\varepsilon}(\varphi)(x',y)=\left\{
 \begin{array}{ll}
 \ds \varphi \Big(\e\Big [\frac{(x', 0)}{\e}\Big]_Y + \e y\Big) & \text{ a.e. for } (x',y) \in \widehat \omega_{\e} \times S^{j},\\
 0 & \text{ a.e. for } (x',y) \in \Lambda_\varepsilon \times S^{j}.
 \end{array}
 \right.
 \end{equation*}
\end{definition}
\begin{remark} If $\phi\in W^{1,p}(\Omega^j_\e)$, $j=1,\ldots,m$, $p\in[1,+\infty]$, $\mathcal T^{bl, j}_{\varepsilon}(\varphi)$ is just the trace of $\mathcal T_{\varepsilon}(\varphi)$ on $\omega \times S^j$.
\end{remark} 
\medskip

 \begin{proposition}[Properties of the operators $\mathcal T_{\varepsilon}$, $\mathcal T^{bl,j}_{\varepsilon}$] \leavevmode
 \label{unfpr}
 \begin{enumerate}
 \item For any $\varphi \in L^1 (\Omega_{\e}^M)$,
 \begin{equation*}
 \int_{\widehat \Omega_{\e}^M} \varphi \, dx = \frac{\e}{|Y'|} \int_{\omega \times Y} \mathcal T_{\e}(\varphi)(x',y) dx' dy.
 \end{equation*}
 
 \item For any $\varphi \in L^2 (\Omega_{\e}^M)$,
 \begin{equation*}
 \|\varphi\|_{L^{2}(\widehat \Omega_{\e}^M)} = \sqrt \e \|\mathcal T_{\e}(\varphi)\|_{L^{2}(\omega_{\e} \times Y)}.
 \end{equation*}
 
 \item Let $\varphi \in H^{1}(\Omega_{\e}^M)$. Then, 
 \begin{equation*}
  \nabla_{y}(\mathcal T_{\varepsilon}(\varphi)) = \varepsilon \mathcal T_{\varepsilon}(\nabla \varphi) \quad \text{ a.e. in } \omega \times Y.
 \end{equation*}
 In a similar way, for any $v \in H^1 (\O_{\e}^M; \mathbb R^3)$
 \begin{equation*}
e_{y}(\mathcal T_{\e}(\varphi))= \e \mathcal T_{\e}(e(\varphi)) \quad \text{ a.e. in } \omega \times Y.
 \end{equation*}

 \item For any $\psi \in L^{p}(S^j_{\e}), p \in [1, +\infty]$
 \begin{equation}
 \label{4.1ij}
 \int_{S^j_{\e}} \psi \, dx = \frac{1}{|Y'|}\int_{\omega_{\e} \times S^j} \mathcal T^{bl, j}_{\e}(\psi)(x',y) dx' d\sigma(y)
 \end{equation} 
 and
 \begin{equation} 
 \label{4.1iij}
\|\psi\|_{L^p(S^j_{\e})}= \frac{1}{|Y'|^{{1/p}}}\|\mathcal T^{bl, j}_{\e}(\psi)\|_{L^p(\omega \times S^{j})}.
\end{equation}
 \end{enumerate}
 \end{proposition}
 
 \begin{proof}
 Proofs for the properties 1-3 can be found in \cite{cior}. For the last property starting from the right-hand side we obtain
 \begin{equation*}
 \begin{aligned}
& \int_{\omega \times S^j} \mathcal T^{bl, j}_{\e} (\psi) (x', y) dx' d\sigma(y) = \int_{\widehat{\omega}_{\e} \times S^j} \mathcal T^{bl, j}_{\e} (\psi)(x', y) dx' d\sigma(y) = \sum_{\xi' \in \Xi_{\e}} \int_{(\e \xi' + \e Y') \times S^j} \mathcal T^{bl, j}_{\e} (\psi) dx' d\sigma(y)\\
&\hskip 2cm = \e^2 \sum_{\xi' \in \Xi_{\e}} \int_{S^j} \psi (\e \xi' + \e s) d\sigma(s) = \sum_{\xi' \in \Xi_{\e}} \int_{(\e \xi' + \e S^j)} \psi (x) d\sigma(x) = \int_{S_{\e}^j} \psi (x) d\sigma(x).
\end{aligned}
 \end{equation*}
\vskip-11mm
 \end{proof}
 \medskip
 
\subsection{Unilateral Korn inequality}
Let ${\cal O}$ be a bounded  open subset  of $\R^3$. We denote
$${\cal R}=\Big\{ r\in H^1({\cal O}; \R^3)\;|\; r(x)= a+b\land x,\;\; (a,b)\in \R^3\times \R^3\Big\}.$$
We recall that a bounded domain ${\cal O}$ satisfies the Korn-Wirtinger inequality if there exists a constant $C_{\cal O}$ such that for every $v\in H^1({\cal O}; \mathbb{R}^3)$ there exists $r\in {\cal R}$ such that
\begin{equation}\label{Eq. 37}
\|v-r\|_{H^1({\cal O};\mathbb{R}^3)}\le C_{\cal O}\|e(v)\|_{L^2({\cal O}; \R^{3\times 3})}.
\end{equation}
A domain like ${\cal O}$ is called a Korn-domain. We equip $H^1({\cal O}; \R^3)$ with the following scalar product
$$<u,v>=\int_{\cal O} e(u):e(v)\, dx+\int_{\cal O} u\cdot v dx.$$
If ${\cal O}$ is a Korn-domain, the associated norm is equivalent to the usual norm of $H^1({\cal O}; \R^3)$.
\begin{definition}For a Korn-domain ${\cal O}$  denote 
$$W^1({\cal O})\doteq\Big\{v\in H^1({\cal O};\mathbb {R}^3)\;\;|\;\; \int_{\cal O} v(x)\cdot r(x) dx=0\;\;\hbox{for all } r\in {\cal R}\Big\}.$$
\end{definition}
 Observe that there exists a constant such that for every $v\in W^1({\cal O})$ we get
$$\|v\|_{H^1({\cal O};\mathbb{R}^3)}\le C\|e(v)\|_{L^2({\cal O}; \R^{3\times 3})}.$$
Considering the orthogonal decomposition $H^{1}({\cal O};\mathbb{R}^3)=W^1({\cal O})\oplus {\cal R}$,  every $v\in H^1({\cal O};\R^3)$ can be written as 
$$ v= (v-r_v)+r_v,\qquad v-r_v\in W^1({\cal O}),\;\; r_v\in {\cal R}.$$
The map $v\longmapsto r_v$ is the orthogonal projection of $v$ on ${\cal R}$. From \eqref{Eq. 37} we get
\begin{equation}\label{Eq. 38}
\|v-r_v\|_{H^1({\cal O};\mathbb{R}^3)}\le C_{\cal O}\|e(v)\|_{L^2({\cal O}; \R^{3\times 3})}.
\end{equation}
We also recall that if ${\cal O}$ is a bounded Lipschitz domain, there exists a constant $C$ such that
\begin{equation}\label{TR}
\forall v \in H^1({\cal O}; \mathbb R^3), \quad \|v\|_{L^2(\partial {\cal O})}\leq C \big(\|e(v)\|_{L^2({\cal O})} + \|v\|_{L^2({\cal O})}\big).
\end{equation}
\smallskip
 
We will use the following proposition from \cite{cior3}.
\begin{proposition} [Unilateral Korn inequality]  \label{newKorn}
If ${\cal O}$ is a bounded Lipschitz domain, there exists a constant $C$ such that
\begin{equation} \label{nonlockeddomain}
\forall v \in H^1({\cal O} ; \R^3), \quad \|v\|_{H^1({\cal O})}\leq C \big(\|e(v)\|_{L^2({\cal O})} + \|(v_{\nu})^{+}\|_{L^1(\partial {\cal O})} + \|v_\tau\|_{L^1(\partial {\cal O})}\big).
\end{equation}
\end{proposition}
\noindent We denote $O^{j}$ the center of gravity of $Y^{j}$, $j=1,\ldots,m$.
\medskip

Let $u$ be in $H^{1}(\O_{\e}^{j} ; \R^3)$ and $r^{j}_u(\e\xi)$, $j=1,\ldots,m$,  the orthogonal projection of $u(\e \xi+\e y)_{|_{Y^{j}}}$ on ${\cal R}$. We write
$$r^{j}_u(\e\xi)(y)=a^{j}(\e\xi)+b^{j}(\e \xi)\land(y-O^{j}),\qquad \hbox{for every } y\in Y^{j}, \;\;\xi\in \Xi_\e.$$ We define the piecewise constant functions $a^{j}_u$ and $b^{j}_u$ by
\begin{equation}
\begin{aligned}
 a^{j}_u(x')&=a^{j}(\e\xi)\quad x'\in \e\xi+\e Y,\;\; \xi\in \Xi_\e,\\
a^{j}_u(x')&=0\qquad x'\in \Lambda_\e,\\
b^{j}_u(x')&=b^{j}(\e\xi)\quad x'\in \e\xi+\e Y,\;\; \xi\in \Xi_\e,\\
b^{j}_u(x')&=0\qquad x'\in \Lambda_\e. 
\end{aligned}
\end{equation}

These functions belongs to $L^\infty(\omega;\R^3)$ and the associated rigid body field, still denoted $r^{j}_u$, belongs to $L^\infty(\omega; {\cal R})$.
\medskip

As a consequence of the above Proposition \ref{newKorn} we get
\begin{proposition} \label{smallinclusions}
There exists a constant $C$ (independent of $\e$) such that for every $j= 1, \ldots, m$ and  for every $u$ in $H^{1}(\O_{\e}^{j} ; \R^3)$, 
\begin{equation}
 \label{smalldomain}
 \begin{aligned}
&\sum_{\xi\in \Xi_\e}\|u-r^{j}_u(\e\xi)\|^2_{L^{2}(\e\xi+\e Y^{j})}+ \e^2 \|\nabla (u-r^{j}_u(\e\xi))\|^2_{L^{2} (\e\xi+\e Y^{j}))} \leq  C \e^2 \|e(u)\|^2_{L^2(\O_{\e}^{j})} ,\\
&\|a^{j}_u\|_{L^1(\omega)}+\e \|b^{j}_u\|_{L^1(\omega)}\le C\sqrt \e \|e(u)\|_{L^2(\O_{\e}^{j})} + C \big(\|(u_{\nu})^{+}\|_{L^{1}(S^{j}_\e)} + \|u_{\tau}\|_{{L^{1}(S^{j}_\e)}}\big),\\
&\|u\|_{L^1(\O_{\e}^{j})}\le C  \e^{3/2} \|e(u)\|_{L^2(\O_{\e}^{j})} + C  \e \big(\|(u_{\nu})^{+}\|_{L^{1}(S^{j}_\e)} + \|u_{\tau}\|_{{L^{1}(S^{j}_\e)}}\big).
\end{aligned}
\end{equation}
\end{proposition}

\begin{proof} Applying \eqref{Eq. 38} (after $\e$-scaling) gives
\begin{equation}
 \label{Eq. 312}
 \|u-r^{j}_u(\e\xi)\|^2_{L^{2}(\e\xi+\e Y^{j})}+ \e^2 \|\nabla (u-r^{j}_u(\e\xi))\|^2_{L^{2} (\e\xi+\e Y^{j})} \leq  C \e^2 \|e(u)\|^2_{L^2(\e\xi+\e Y^{j})}.
 \end{equation}
Then adding the above inequalities yields  \eqref{smalldomain}$_1$.

We have
$$\|r^{j}_u(\e\xi)\|^2_{L^2(\e\xi+\e Y^{j})}\le 2\big(\|u-r^{j}_u(\e\xi)\|^2_{L^2(\e\xi+\e Y^{j})}+\|u\|^2_{L^2(\e\xi+\e Y^{j})}\big)$$
which leads to
$$|a^{j}_u(\e\xi)|^2\e^3+|b^{j}_u(\e\xi)|^2\e^5  \le C\big(\|u-r^{j}_u(\e\xi)\|^2_{L^2(\e\xi+\e Y^{j})}+\|u\|^2_{L^2(\e\xi+\e Y^{j})}\big).$$
Taking into account \eqref{Eq. 312} and \eqref{nonlockeddomain} (after $\e$-scaling), we get
$$ |a^{j}_u(\e\xi)|\e^2+|b^{j}_u(\e\xi)|\e^3  \le C\e^{3/2}\|e(u)\|_{L^2(\e\xi+\e Y^{j})}+C \big(\|(u_{\nu})^{+}\|_{L^1(\e\xi+\e S^{j})} + \|u_\tau\|_{L^1(\e\xi+\e S^{j})}\big).$$
Adding these inequalities gives
$$\begin{aligned}
\|a^{j}_u\|_{L^1(\omega)}+\e \|b^{j}_u\|_{L^1(\omega)} & \le C \e^{3/2}\sum_{\xi\in \Xi_\e}\|e(u)\|_{L^2(\e\xi+\e Y^{j})}+C\big(\|(u_{\nu})^{+}\|_{L^1(S^{j}_\e)} + \|u_\tau\|_{L^1(S^{j}_\e)}\big)\\
& \le C \e^{3/2}\Big(\sum_{\xi\in \Xi_\e} 1^2\Big)^{1/2}\Big(\sum_{\xi\in \Xi_\e}\|e(u)\|^2_{L^2(\e\xi+\e Y^{j})}\Big)^{1/2}+C\big(\|(u_{\nu})^{+}\|_{L^1(S^{j}_\e)} + \|u_\tau\|_{L^1(S^{j}_\e)}\big)\\
& \le C \sqrt\e \|e(u)\|_{L^2(\O_{\e}^{j})}+C\big(\|(u_{\nu})^{+}\|_{L^1(S^{j}_\e)} + \|u_\tau\|_{L^1(S^{j}_\e)}\big).
\end{aligned}$$
Finally, estimate \eqref{smalldomain}$_3$ is an immediate consequence of \eqref{smalldomain}$_1$ and \eqref{smalldomain}$_2$.
\end{proof}
\begin{remark} Due to \eqref{Eq. 312} and \eqref{nonlockeddomain} (again after $\e$-scaling), we also obtain\begin{equation}\label{L2R}
\|u\|_{L^2(\O_{\e}^{j})}\le C \e \|e(u)\|_{L^2(\O_{\e}^{j})} + {C\over \sqrt  \e} \big(\|(u_{\nu})^{+}\|_{L^{1}(S^{j}_\e)} + \|u_{\tau}\|_{{L^{1}(S^{j}_\e)}}\big),\qquad j= 1, \ldots, m.
\end{equation}
\end{remark}

\subsection{Korn inequality for the perforated layer}
Now we want to derive the Korn inequality for the simply connected part of the layer. 

Denote
$$H^1_\Gamma(\O^*_\e)=\Big\{ \phi\in H^1(\O^*_\e)\;|\; \phi=0\;\hbox{a.e. on } \Gamma\Big\}.$$
\begin{proposition}
\label{prunk} There exists a constant $C$ independent of $\e$ such that for every $u$ in $H^1_\Gamma (\Omega^*_{\e} ; \R^3)$
\begin{equation}
\label{eq315}
\| u \|_{H^1 (\O_{\e}^* )} \leq C \| e(u) \|_{L^2 (\O_{\e}^*)}.
\end{equation}
We also have
\begin{equation}
\label{Eq 315-bis}
\begin{aligned}
\| \nabla u \|^2_{L^2(\O_{\e}^{a} )} +\e \| \nabla u \|^2_{L^2(\O_{\e}^0 )}+\| \nabla u \|^2_{L^2(\O^{b} )}\leq C\big( \| e(u) \|^2_{L^2 (\O_{\e}^{a})}+\e  \| e(u) \|^2_{L^2 (\O_{\e}^0)}+ \| e(u) \|^2_{L^2 (\O^{b})}\big),\\
\|u \|^2_{L^2(\O_{\e}^{a} )} +{1\over \e} \|u \|^2_{L^2(\O_{\e}^0 )}+\|u \|^2_{L^2(\O^{b} )}\leq C\big( \| e(u) \|^2_{L^2 (\O_{\e}^{a})}+\e  \| e(u) \|^2_{L^2 (\O_{\e}^0)}+ \| e(u) \|^2_{L^2 (\O^{b})}\big).
\end{aligned}
\end{equation}
\end{proposition}

\begin{proof}
\noindent{\it Step 1. } First, we construct an ''extension'' of  $u$. Set
$$Y_{\eta}\doteq\big  \{y\in Y\;| \;  \hbox{dist}(y, \partial Y )<\eta/2\big\}. $$
The domain $Y_{\eta}$ is a bounded domain with a Lipschitz boundary. Therefore  there is an extension operator $P_\eta$ from $H^{1}(Y_{\eta})$ into $H^{1}(Y)$ and a constant $C$ (which depends on $\eta$) such that (see  \cite{CiDaGr2})
\begin{equation}\label{eq.311}
\forall v\in H^1(Y_{\eta}) \qquad  \|P_\eta(v) \|_{L^{2}(Y)}\leq C  \|v \|_{L^{2}(Y_{\eta})} \text{ \quad and \quad}  \|\nabla_y P_\eta(v) \|_{L^{2}(Y)}\leq C  \|\nabla_y v \|_{L^{2}(Y_{\eta})}.
\end{equation} 
Let $w\in [H^1(Y_{\eta})]^3$ and $r_w$  the projection of $w$ on $\cal R$, we have
\begin{equation} \label{eq.312}
\|w-r_{w}\|_{H^{1}(Y_{\eta})}\leq C \|e_y(w)\|_{L^{2}(Y_{\eta})}.
\end{equation}
The constant depends on $\eta$.
\medskip

Now for every  $w\in [H^1(Y_{\eta})]^3$ we define the extension  $Q_\eta(w)\doteq P_\eta(w-r_{w}) + r_{w}$ of $w$. From \eqref{eq.311} and \eqref{eq.312} we get
\begin{equation} \label{4.4}
Q_\eta(w)\in [H^1(Y)]^3,\qquad  \|e_y(Q_\eta(w)) \|_{L^{2}(Y)}\leq C \|e_y(w)\|_{L^{2}(Y_{\eta})}.
\end{equation}
Applying the above result to the restriction of the displacement $y\longmapsto u(\e\xi+\e y)$ to the cell $Y_\eta$, $\xi\in \Xi_\e$,  allows to define an extension $\widetilde{u}$ of $u$ in the layer $\widehat{\Omega}^M_\e$. Estimate \eqref{4.4} leads to
\begin{equation*}
\widetilde{u}\in H^1(\widehat{\O}^M_\e; \R^3),\qquad \|e(\widetilde{u}) \|^2_{L^{2}(\widehat{\O}^M_\e)}\leq C\sum_{\xi\in \Xi_\e} \|e(u)\|^2_{L^{2}(\e(\xi+Y_{\eta}))}\le C \|e(u)\|^2_{L^{2}(\O^0_\e)}.
\end{equation*}
The constants do not depend on $\e$.

We set $\widetilde{u}=u$ in $\O\setminus \overline{\widehat{\O}^M_\e}$.
The displacement $\widetilde{u}$ belongs to $H^1(\O; \R^3)$, it vanishes on $\Gamma$ and it satisfies
\begin{equation}
\label{eq.315}
 \|e(\widetilde{u}) \|_{L^{2}(\O^M_\e)}\leq  C \|e(u)\|_{L^{2}(\O^0_\e)},\qquad  \|e(\widetilde{u}) \|_{L^{2}(\O)}\leq  C \|e(u)\|_{L^{2}(\O^*_\e)}.
\end{equation}
The constant do not depend on $\e$.

\noindent{\it Step 2.} From the Korn's inequality, the hypothesis that the measure of $\Gamma$ is positive  and \eqref{eq.315}$_2$ we obtain
\begin{equation}
\label{kineq}
\| \widetilde{u} \|_{H^1(\O)} \leq C \| e(\widetilde{u}) \|_{L^2(\O)}\leq C \| e(u) \|_{L^2(\O^*_\e)}.
\end{equation}

\noindent{\it Step 3.} We prove \eqref{eq315}. %\noindent{\it Step 1.} 
Since by construction the domains $Y^0$  is a Korn-domain, there exists a constant $C>0$  such that for every $v\in [H^1(Y^0)]^3$ equal to zero on  $\partial Y$
\begin{equation*}
\| v \|_{H^1 (Y^0)} \leq C  \| e_y(v) \|_{L^2 (Y^0)}.
\end{equation*}
Applying the above result to the restriction of the displacement $y\longmapsto (u-\widetilde{u})(\e\xi+\e y)$ to the cell $Y^0$, $\xi\in \Xi_\e$, gives
$$\|(u-\tilde u)\|^2_{L^2(\widehat{\O}^M_\e)}\le C\e^2\sum_{\xi\in \Xi_\e}\|e(u-\tilde u)\|^2_{L^{2}(\e(\xi+Y^0))},\qquad \|\nabla (u-\tilde u)\|^2_{L^2(\widehat{\O}^M_\e)}\le C\sum_{\xi\in \Xi_\e}\|e(u-\tilde u)\|^2_{L^{2}(\e(\xi+Y^0))}.$$ The constants do not depend on $\e$. Hence, using the fact that $u-\widetilde{u}$  vanishes in $\Omega^0_\e \setminus \overline{\widehat{\O}^M_\e}$ and due to estimate \eqref{kineq}, we obtain
 \begin{equation}\label{EQ. 322-0}
 \|(u-\tilde u)\|_{L^2(\O^0_\e)}\le C\e\|e(u)\|_{L^{2}(\O^0_\e)},\qquad \|\nabla (u-\tilde u)\|_{L^2(\O^0_\e)}\le C\|e(u)\|_{L^{2}(\O^0_\e)}.
 \end{equation}
 Combining the above inequalities and \eqref{kineq} gives \eqref{eq315}.
\smallskip
 
 \noindent{\it Step 4.} We prove \eqref{Eq 315-bis}. The Korn inequality and the trace theorem give
 \begin{equation}\label{EQ. 322}
 \begin{aligned}
&\|\tilde u\|_{L^2(\O^{b})}+\|\tilde{u}\|_{L^2(\Sigma)}+\|\nabla \tilde u\|_{L^2(\O^{b})}\le C\|e(\tilde u)\|_{L^2(\O^{b})},\\
&\|\tilde u\|_{L^2(\O^{a}_\e)}+\|\nabla \tilde u\|_{L^2(\O^{a}_\e)}\le C\|e(\tilde u)\|_{L^2(\O^{a}_\e)}+C\|\tilde{u}\|_{L^2(S^{a}_\e)}.
\end{aligned}
\end{equation}
Besides we have
\begin{equation}\label{EQ. 323}
\begin{aligned}
&\|\tilde u\|^2_{L^2(S^{a}_\e)}\le C\big(\|\tilde{u}\|^2_{L^2(\Sigma)}+\e \|\nabla \tilde u\|^2_{L^2(\O^M_\e)}\big),\\
&\|\tilde u\|^2_{L^2(\O^M_\e)}\le C\big(\e \|\tilde{u}\|^2_{L^2(\Sigma)}+\e^2 \|\nabla \tilde u\|^2_{L^2(\O^M_\e)}\big).
\end{aligned}
\end{equation} Taking into account the above estimates \eqref{EQ. 322}-\eqref{EQ. 323} together with \eqref{kineq}-\eqref{EQ. 322-0} we obtain \eqref{Eq 315-bis}.
\end{proof}
We set
\begin{equation}\label{Ve}
{\V}_\e\doteq \Big \{ {\bf v}=(v, v^1 \ldots, v^{m})\;\big|\; {\bf v}\in H^1_\Gamma(\O_{\e}^* ; \R^3 ) \times H^{1}(\O^{1}_{\e} ; \R^3 )\times\ldots\times H^{1}(\O^{m}_{\e} ; \R^3 )\Big\}.
\end{equation} We will denote by $[{\bf v}]_{S^{j}_{\e}}$ the jump of the vector field across the surface $S^{j}_{\e}$, $j=0,\ldots, m$. 

\noindent More precisely, for $j=1,\ldots,m$ we set   $[{\bf v}]_{S^{j}_{\e}}= (v^{j}- v)_{|S^{j}_{\e}}, \;\; [{\bf v}_\nu]_{S^{j}_{\e}}= (v^{j}- v)\cdot\nu_{|S^{j}_{\e}}, \;\; [{\bf v}_\tau]_{S^{j}_{\e}}= [{\bf v}]_{S^{j}_{\e}}- [{\bf v}_\nu]_{S^{j}_{\e}}$ and we define $[{\bf v}]_{S^{0}_{\e}}$, $[{\bf v}_\nu]_{S^{0}_{\e}}$, and $[{\bf v}_\tau]_{S^{0}_{\e}}$ by
\begin{gather}\label{JS0}
[{\bf v}]_{S^{0}_{\e}}(x^{'})=\lim_{t\to 0, t>0} v\big(x^{'}+t\nu(x^{'})\big)-v\big(x^{'}-t\nu(x^{'})\big),\\
[{\bf v}_{\nu}]_{S^{0}_{\e}} (x^{'})= \lim_{t\to 0, t>0} \Big(v\big(x^{'}+t\nu(x^{'})\big)-v\big(x^{'}-t\nu(x^{'})\big)\Big) \cdot \nu(x^{'}),\qquad \hbox{for a.e. } x^{'}\in S^{0}_{\e}\\
[{\bf v}_\tau]_{S^0_{\e}}= [{\bf v}]_{S^0_{\e}}- [{\bf v}_\nu]_{S^0_{\e}}.
\end{gather}
We equip ${\V}_\e$ with the following norm$^{(}$\footnotemark$^{)}$:
\footnotetext{ {\it Here we consider the case where the layer and the inclusions are made of a soft material.
}}
\begin{equation*}
\forall {\bf v}\in {\V}_\e,\qquad ||{\bf v}||_{\V_\e} \doteq\sqrt{\|\nabla v\|^2_{L^2(\O^{b};\R^{3\times 3} )}+\|\nabla v\|^2_{L^2(\O^{a}_\e;\R^{3\times 3}  )}+\e\|\nabla v\|^2_{L^2(\O_{\e}^0;\R^{3\times 3} )}+\sum_{j=1}^m\e\|v^{j}\|^2_{H^1(\O_{\e}^{j} ;\R^{3} )}}.
\end{equation*}
With the above norm, ${\V}_\e$ is a Hilbert space.
\smallskip

In order to measure the elements of $\V_\e$, we defined the following two maps:
\begin{equation*}
\begin{aligned}
\forall {\bf v}\in {\V}_\e,\qquad {\bf M}({\bf v}) \doteq&\sqrt{\|e(v)\|^2_{L^2(\O^{b} )}+\|e(v)\|^2_{L^2(\O^{a}_\e )}+\e\|e(v)\|^2_{L^2(\O_{\e}^0)}+\sum_{j=1}^m\e\|e(v^{j})\|^2_{L^2(\O_{\e}^{j} )}}\\
&\;\;\;+\sum_{j=1}^m \big(\|[{\bf v}_\nu]^{+}_{S^{j}_{\e}}\|_{L^1(S^{j}_\e)} + \|[{\bf v}_\tau]_{S^{j}_{\e}}\|_{L^1(S^{j}_\e)}\big).
\end{aligned}
\end{equation*}
Observe that ${\bf M}(\cdot)$ is not a norm nor a semi-norm.
Denote also
 $$\begin{aligned}
 \forall {\bf v}\in {\V}_\e,\qquad {\bf E}({\bf v}) \doteq& \|e(v)\|^2_{L^2(\O^{b} )}+\|e(v)\|^2_{L^2(\O^{a}_\e )}+\e\|e(v)\|^2_{L^2(\O_{\e}^0)}+\sum_{j=1}^m\e\|e(v^{j})\|^2_{L^2(\O_{\e}^{j} )}\\
 &\;\;\;+\sum_{j=1}^m \big(\|[{\bf v}_\nu]^{+}_{S^{j}_{\e}}\|_{L^1(S^{j}_\e)} + \|[{\bf v}_\tau]_{S^{j}_{\e}}\|_{L^1(S^{j}_\e)}\big).
 \end{aligned}$$
 The quantity ${\bf E}({\bf v})$ is a kind of energy, it can replace the total elastic energy of the system. Observe that
 $$\forall {\bf v}\in {\V}_\e,\qquad {\bf M}({\bf v})\leq \sqrt{{\bf E}({\bf v}) }+ {\bf E}({\bf v}) \qquad {\bf E}({\bf v}) \leq {\bf M}({\bf v})^2+{\bf M}({\bf v}).$$ If we consider elements satisfying
 $${\bf M}({\bf v})\leq C_0$$ then 
 \begin{equation}
 \label{norm-en}
  \sqrt{{\bf E}({\bf v}) }\leq (1+C_0){\bf M}({\bf v}), \qquad {\bf M}({\bf v})\leq (1+C_0)^2\sqrt{{\bf E}({\bf v}) }.
 \end{equation}
 In the same way, if we have
 $${\bf E}({\bf v})\leq C^2_1$$ then 
 $$ {\bf M}({\bf v})\leq (1+C_1)\sqrt{{\bf E}({\bf v}) },\qquad \sqrt{{\bf E}({\bf v}) }\leq (1+C_1)^2{\bf M}({\bf v}).$$

Below we summarize the estimates for ${\bf v}\in {\V}_\e$.
\begin{proposition}
\label{prk}
There exists a constant $C$ independent of $\e$ such that for all ${\bf v}\in {\V}_\e$
\begin{equation}
\begin{aligned}
\label{eqk}
\| \nabla v \|^2_{L^2(\O_{\e}^{a} )} +\e \| \nabla v \|^2_{L^2(\O_{\e}^0 )}+\| \nabla v \|^2_{L^2(\O^{b} )}& \leq C[{\bf M}({\bf v})]^2,\\
\|v \|^2_{L^2(\O_{\e}^{a} )} +{1\over \e} \|v \|^2_{L^2(\O_{\e}^0 )}+\|v \|^2_{L^2(\O^{b} )}& \leq C[{\bf M}({\bf v})]^2,\\
\sum_{\xi\in \Xi_\e}\|v^{j}-r^{j}_{\bf v}(\e\xi)\|^2_{L^{2}(\e\xi+\e Y^{j})}+ \e^2 \|\nabla (v^{j}-r^{j}_{\bf v}(\e\xi))\|^2_{L^{2} (\e\xi+\e Y^{j}))}& \leq  C \e[{\bf M}({\bf v})]^2,\\
\sum_{j=1}^m\big(\|a^{j}_{\bf v}\|_{L^1(\omega)}+\e \|b^{j}_{\bf v}\|_{L^1(\omega)}\big)\le C  {\bf M}({\bf v}),\qquad \sum_{j=1}^m\|v^{j}\|_{L^1(\O_{\e}^{j})}& \le C  \e {\bf M}({\bf v}).
\end{aligned}
\end{equation}
The constants do not depend on $\e$.
\end{proposition}
\begin{proof} Estimates \eqref{eqk}$_1$, \eqref{eqk}$_2$ and \eqref{eqk}$_3$  are the immediate consequences of \eqref{Eq 315-bis}, \eqref{smalldomain}$_1$. Then from \eqref{TR} (after $\e$-scaling) we get
$$\sum_{j=1}^m\|v\|_{L^2(S^{j}_\e)}\le {C\over \sqrt \e}\big(\e\|e(v)\|_{L^2(\O^0_\e)}+\|v\|_{L^2(\O^0_\e)}\big).$$ Hence
$$\sum_{j=1}^m \big(\|[{\bf v}_\nu]^{+}_{S^{j}_{\e}}\|_{L^1(S^{j}_\e)} + \|[{\bf v}_\tau]\|_{L^1(S^{j}_\e)}\big) \le C  {\bf M}({\bf v}).$$
The last two estimates in \eqref{eqk} follow from the one above  and \eqref{smalldomain}$_2$.
\end{proof} 
\begin{remark} From \eqref{L2R}  we also obtain
\begin{equation*}
\|v^{j}\|_{L^2(\O_{\e}^{j})}\le {C\over \sqrt \e}  {\bf M}({\bf v}),\qquad j= 1, \ldots, m.
\end{equation*} Hence
\begin{equation}\label{L2R-Final}
\|{\bf v}\|_{\V_\e}\le C{\bf M}({\bf v}).
\end{equation} 
The constants do not depend on $\e$.
\end{remark}

\section{Convergence results}
Every $v\in H^1(\Omega^{a}_\e;\R^3)$ is extended by reflexion in a displacement belonging to $H^1(\omega\times (\e, 2L-\e);\R^3)$. Denote
\begin{equation}
\begin{aligned}
H^1_{per}(Y^0 )=\big\{\phi\in H^1(Y^0 )\;\; |\;\; &\phi(0,y_2,y_3)=\phi(1,y_2,y_3)\enskip \hbox{for a.e. } (y_2,y_3)\in (0,1)^2,\\
&\phi(y_1,0,y_3)=\phi(y_1, 1, y_3)\enskip \hbox{for a.e. } (y_1,y_3)\in (0,1)^2 \big\},\\
&H^1_\Gamma(\O^{b})=\Big\{ \phi\in H^1(\O^{b})\;|\; \phi=0\;\hbox{a.e. on } \Gamma\Big\}.
\end{aligned}
\end{equation}
Before giving the convergence results, we prove the following  lemma of homogenization: 
\begin{lemma}\label{lem4.4.1}
Let $\{\phi_\e\}_\e$ be a sequence in  $H^1(\Omega)$ satisfying
\begin{equation*}
\| \phi_\e \|^2_{H^1(\O_{\e}^{a} )} +\| \phi_\e \|^2_{H^1(\O^{b} )}+\e \| \nabla \phi_\e \|^2_{L^2(\O_{\e}^M )} +{1\over \e}\|\phi_\e\|^2_{L^2(\Omega_\e^M)}\leq C\end{equation*} where the constant $C$ does not depend on $\e$. There exist a subsequence -still denoted $\e$- and $\phi^{b}\in H^1(\Omega^{b})$,  $\phi^{a}\in H^1(\Omega^{a})$,  $\widehat{\phi}\in L^2(\omega ; H^1_{per}(Y))$ such that
\begin{equation}\label{Eq4.2}
\begin{aligned}
\phi_\e & \rightharpoonup \phi^{b}\quad \hbox{weakly in } H^1(\O^{b}),\\
\phi_\e(\cdot +\e {\bf e}_3) & \rightharpoonup \phi^{a}\quad \hbox{weakly in } H^1(\O^{a}),\\
{\cal T}_\e(\phi_\e) & \rightharpoonup \widehat{\phi}\quad \hbox{weakly in } L^2(\omega; H^1(Y)).\end{aligned}
\end{equation}
Moreover, we have
\begin{equation}\label{CondInterface-0}
\phi^{b}(x', 0)=\widehat{\phi}(x', y_1,y_2,0),\qquad \phi^{a}(x', 0)=\widehat{\phi}(x', y_1,y_2,1)\quad \hbox{for a.e. } (x',y_1,y_2)\in \omega\times Y^{'}.
\end{equation}
\end{lemma}
\begin{proof} The function  $\phi_\e$ is extended by reflexion in a function  belonging to $H^1(\omega\times (0, 2L-\e))$ in order to obtain convergence \eqref{Eq4.2}$_2$. 

We only prove the first equality in \eqref{CondInterface-0}, the second one is obtained in the same way. 
Consider the function defined by
$$\overline{\phi}_\e(x_1,x_2,x_3)=\phi_\e(x_1,x_2,x_3)-\phi_\e (x_1,x_2,-x_3)\qquad x=(x_1,x_2,x_3)\in \Omega^M_\e.$$ It satisfies
$$\| \nabla \overline{\phi}_\e \|^2_{L^2(\O_{\e}^M )} \leq \| \nabla \phi_\e \|^2_{L^2(\O_{\e}^M )} +\| \nabla \phi_\e \|^2_{L^2(\O^{b} )} \leq {C\over \e},\qquad \overline{\phi}_\e=0\quad \hbox{on}\;\; \omega\times\{0\}.$$ Hence
$\| \overline{ \phi}_\e \|^2_{L^2(\O_{\e}^M )} \leq C\e$. Due to the convergences \eqref{Eq4.2}$_1$ and \eqref{Eq4.2}$_3$ we have
$${\cal T}_\e(\overline{\phi}_\e) \rightharpoonup \widehat{\phi}-\phi^{b}_{|\omega\times\{0\}}\quad \hbox{weakly in } L^2(\omega; H^1(Y)).$$ Since the trace of the function $y\longmapsto {\cal T}_\e(\overline{\phi}_\e)(x^{'},y)$ on the face $Y^{'}\times \{0\}$ vanishes for a.e. $x^{'}\in \omega$ the result is proved.
\end{proof}

\begin{theorem} 
\label{convergence}
Let $\{{\bf v}_\e\}_\e$ be a sequence in  ${\V}_\e$ satisfying
\begin{equation}\label{EQ. 400}
{\bf M}({\bf v}_\e)\le C
\end{equation} where the constant $C$ does not depend on $\e$. There exist a subsequence -still denoted $\e$- and $v^{b}\in H^1_\Gamma(\Omega^{b};\R^3)$,  $v^{a}\in H^1(\Omega^{a};\R^3)$,  $\widehat{v}^0\in L^2(\omega ; H^1_{per}(Y^0 ; \R^3))$,  $\widehat{v}^{j}\in L^2(\omega ; H^1(Y^j; \R^3))$, $a^{j}\in {\cal M}(\omega ; \R^3)$ and $b^{j}\in {\cal M}(\omega ; \R^3)$, ($j=1,\ldots, m$), such that
\begin{equation}\label{Eq4.3}
\begin{aligned}
v_\e & \rightharpoonup v^{b}\quad \hbox{weakly in } H^1_\Gamma(\O^{b};\R^3),\\
v_\e(\cdot +\e {\bf e}_3) & \rightharpoonup v^{a}\quad \hbox{weakly in } H^1(\O^{a};\R^3),\\
{\cal T}_\e(v_\e) & \rightharpoonup \widehat{v}^0\quad \hbox{weakly in } L^2(\omega; H^1(Y^0;\R^3)),\\
\e {\cal T}_{\e}(e (v_{\e}) ) &\rightharpoonup e_{y} (\widehat{v}^0) \quad \hbox{weakly  in } L^2(\omega\times Y^0; \R^{3\times 3}),\\
{\cal T}_\e(v^{j}_\e-r^{j}_{\bf v_\e}) & \rightharpoonup \widehat{v}^{j}\quad \hbox{weakly in } L^2(\omega; H^1(Y^{j};\R^3)),\\
{\cal T}_\e(v^{j}_\e) & \rightharpoonup \widehat{v}^{j}+r^{j}\quad \hbox{weakly-* in } {\cal M}(\omega\times Y^{j};\R^3),\\
\e {\cal T}_{\e}(e(v^{j}_{\e})) &\rightharpoonup e_{y}(\widehat{v}^j)\quad \text {weakly  in  }L^2(\omega \times Y^{j}; \R^{3\times 3}),\\
a^{j}_{\bf v_\e} & \rightharpoonup a^{j} \quad \hbox{weakly-* in}\; \; {\cal M}(\omega ; \R^3),\\
b^{j}_{\bf v_\e} & \rightharpoonup b^{j} \quad \hbox{weakly-* in}\; \; {\cal M}(\omega ; \R^3),
\end{aligned}
\end{equation}
where 
$$\forall y\in Y^{j},\qquad r^{j}(\cdot, y)=a^{j}+b^{j}\land (y-O^{j})\quad \hbox{in } {\cal M}(\omega\; ;\; \R^3).$$
Moreover we have
\begin{equation}\label{CondInterface}
 v^{b}(x', 0)=\widehat{v}^0(x', y_1,y_2,0),\qquad v^{a}(x', 0)=\widehat{v}^0(x', y_1,y_2,1)\quad \hbox{for a.e. } (x',y_1,y_2)\in \omega\times Y^{'}.
\end{equation}

\noindent Furthermore
\begin{equation}\begin{aligned}
\label{furthermore}
%{\cal T}_{\e}^{bl, 0}([{\bf v}_{\e}]_{S^{0}_{\e}})&\rightharpoonup [\widehat{v}^0]_{S^{0}}\quad  \hbox{weakly  in } L^2(\omega; H^{1/2}(S^{0};\R^3)),\\
{\cal T}_{\e}^{bl, j}([{\bf v}_{\e}]_{S^{j}_{\e}})&\rightharpoonup (\widehat{v}^j + r^j - \widehat{v}^0)_{|S^{j}}  \quad \text{weakly $*$  in }  {\cal M}(\omega\times S^{j};\R^3),\qquad j=1,\ldots,m.
\end{aligned}
\end{equation}
\end{theorem}
\begin{proof} The convergences \eqref{Eq4.3} are the immediate consequences of the theorems in \cite{ciorp} and \cite{cdg3}. To prove \eqref{CondInterface} we apply the Lemma \ref{lem4.4.1} with the fields of displacements $\widetilde{v}_\e$ introduced in Step 1 of the  Proposition \ref{prunk} proof.
\end{proof}
%Further we will use the notation 
%$$[\widehat{v}^j]_{S^{j}} = (\widehat{v}^j + r^j - \widehat{v}^0)_{S^{j}}.$$
\section{The contact problem for fixed $\e$}\label{S5}

Assume we are given the following symmetric bilinear form on ${\V}_\e$: 
\begin{equation*}
\forall ({\bf u},{\bf v})\in \big({\V}_\e\big)^2,\qquad {\mathbf A}^{\e}({\bf u}, {\bf v})\doteq\sum_{j=1}^{m} \int_{\O_{\e}^{j}} a^{\e} e(u^{j}) : \,e(v^{j}) \,dx + \int_{\O_{\e}^*} a^{\e} e(u) : \,e(v)\,dx,
\end{equation*}
where 
\begin{equation}
a^{\e} (x) = \left\{
 \begin{array}{ll}
  a^{a} (x) & \hbox{for a.e. }\; x \in \Omega^{a}_{\e},\\
  \e a^M_{\e} (x) & \hbox{for a.e. }\; x \in \Omega^M_{\e},\\
  a^{b} (x) & \hbox{for a.e. }\;  x \in \Omega^{b}.
 \end{array}
\right.
\end{equation}
The tensor fields $a_{\e}^M $, $a^{a}$, $a^{b}$ have the following properties.
\begin{itemize}
\item Symmetry: 
$$a^{\e} (x) \eta : \xi =a^{\e} (x) \xi : \eta \qquad \text{a.e. } x \in \Omega, \qquad \forall \xi, \eta \in \mathbb R^{3 \times 3}.$$
\item Boundedness: $a^{\e}$ belongs to $L^\infty(\O; \mathbb R^{3\times 3 \times 3 \times 3})$ and
$$\|a^M_{\e}\|_{L^\infty(\O^M_\e; \mathbb R^{3\times 3 \times 3 \times 3})}+ \|a^{a}\|_{L^\infty(\O^{a}; \mathbb R^{3\times 3 \times 3 \times 3})}+\|a^{b}\|_{L^\infty(\O^{b}; \mathbb R^{3\times 3 \times 3 \times 3})} \leq C.$$ The constant does not depend on $\e$.
\item Coercivity (with constant $\overline {\mathbf\alpha }>0$ independent of $\e$):
\begin{equation}\label{coer}
\begin{aligned}
&\overline \alpha  \;\eta : \eta \leq a_{\e}^M(x) \, \eta : \eta \;\; \hbox{for a.e. } x\in \O^M_\e,\\
&\overline \alpha \;\eta : \eta \leq a^{a}(x) \, \eta : \eta \;\; \hbox{for a.e. } x\in \O^{a},\\
&\overline \alpha \;\eta : \eta \leq a^{b}(x) \, \eta : \eta \;\; \hbox{for a.e. } x\in \O^{b},
\end{aligned}
\qquad \forall \eta \in \R^{3 \times 3}.
\end{equation}
\end{itemize}

Let ${\K}_\e$ be the convex set defined, for non negative functions $g^{j}_{\e}$ belonging to $L^1(S^{j}_{\e})$, $j=0,\ldots, m$, by
\begin{equation}
\label{kaepsilon}
{\K}_\e\doteq \Big \{  {\bf v}\in {\V}_\e, \quad [{\bf v}_{\nu}]_{S^{j}_\e}\leq g^{j}_{\e} \; \text{ on } S^{j}_{\e},\;\; j=0,\ldots, m\Big\}.
\end{equation}
The vector fields ${\bf v}\in \V_\e$ are the admissible deformation fields with respect to the reference configuration $\O_{\e}$.  By standard trace theorems, the jumps belong to $H^{1/2}(S^{j}_{\e})$.
The tensor field 
\begin{equation*}
\sigma^{\e} ({\bf v})\doteq a^{\e} e({\bf v}) \qquad \text{ in } \O_{\e}
\end{equation*}
is the stress tensor associated to the deformation $v$ (not to be confused with the surface measures $d\sigma$!).
\medskip

The  functions $g^{0}_{\e}$ and the $g^{j}_{\e}$'s are the original gaps (in the reference configuration), and the corresponding inequalities in the definition of ${\K}_\e$ represent the non-penetration conditions. In case there is contact in the reference configuration, these functions are just 0.
\smallskip

Consider also the family of convex maps $\Psi^{j}_{\e}, \; 0\leq j\leq m$,  where $\Psi^{j}_{\e} $  is non negative, continuous on $L^1(S^{j}_{\e})$  and satisfies   
\begin{equation}
\label{Mj}
\begin{aligned}
w\in L^1(S^j_\e),\qquad &M^{j}_{\e} \;\|w\|_{L^{1}(S^{j}_{\e})} -a^j_\e \leq \Psi^{j}_{\e}(w)\leq M^{'j}_{\e} \;\|w\|_{L^{1}(S^{j}_{\e})}
\\
& \text{for non negative real numbers }M^{'j}_{\e},\; M^{j}_{\e},\;\; a^j_\e,\quad M^{j}_{\e}\not=0,\;\; M^{'j}_{\e}\not=0.
\end{aligned}
\end{equation}

In case of Tresca friction, the maps $\Psi^{j}_{\e}$ are explicitly given by
\begin{equation}\label{Tresca explicit}
\Psi^{j}_{\e}(w)\doteq \int_{S^{j}_{\e}}G^{j}_{\e}(x) |w(x)|\, d\sigma(x), \qquad G^{j}_{\e}\in L^\infty(S^j_\e),\quad w\in L^1(S^j_\e)
\end{equation} 
with $G^{j}_{\e}$ bounded from below by $M^{j}_{\e}>0 $ for $j=0,\ldots, m$.
\medskip

{\bf Problem ${\cal P}_{\e}$:} Given ${\bf f}_{\e} = (f, f^{1}_{\e}, \cdots, f^{m}_{\e})$ in $L^2(\O) \times L^{\infty}(\O^{1}_{\e})\times\ldots\times L^{\infty}(\O^{m}_{\e})$ find a minimizer over ${\K}_\e$ of the functional 
\begin{equation} \label{problem}
{\cal E}_{\e} ({\bf v})\doteq  \frac{1}{2} \,{\mathbf A}^{\e}({\bf v},{\bf v})+ \sum_{j=0}^{m}\Psi^{j}_{\e}([{\bf v}_{\tau}]_{S^{j}_{\e}}) -\int_{\O_{\e}} {\bf f}_{\e} \cdot {\bf v} \,dx.
\end{equation}

From the properties of convexity of the $\Psi^{j}_{\e}, \, j=0, \ldots, m$, the solutions of ${\cal P}_{\e}$ are the same as that of the following problem:
\medskip

{\bf Problem $\mathcal P^{'}_{\e}$}: Find ${\bf u}_{\e}\in \K_{\e}$ such that for every ${\bf v}\in \K_{\e}$,
\begin{equation}
\label{wf}
 {\mathbf A}^{\e}({\bf u}_{\e},{\bf v}- {\bf u}_{\e})+ \sum_{j=0}^{m}\big(\Psi^{j}_{\e}([{\bf v}_{\tau}]_{S^{j}_{\e}})- \Psi^{j}_{\e}([({\bf u}_{\e})_{\tau}]_{S^{j}_{\e}})\big) \geq\int_{\O_{\e}}{\bf f}_{\e} \cdot ({\bf v} - {\bf u}_{\e})\,dx.
\end{equation}

The strong formulation of the problem is (with $\sigma^\e$ for the stress tensor $\sigma^\e({\bf u}_\e)$):
\begin{gather}
\label{sf}
\left\{
\begin{array}{l}
 \nabla \cdot \sigma^{\varepsilon} = -{\bf f}_{\e} \quad \text{ in } \Omega_{\e}, \\ [1mm]
 \sigma^{\varepsilon}(\nu)_{\nu} \leq 0,\\ [1mm]
 \sigma^{\e}(\nu)_{\nu}\big([({\bf u}_{\e})_{\nu}]_{S^{j}_{\e}}-g^{j}_{\e}\big)=0 \\
 \sigma^{\e}(\nu)_{\tau }\in \partial \Psi^{j}_{\e} ([({\bf u}_{\e})_{\tau}]_{S^{j}_{\e}}) \quad\text{ on } S^{j}_{\e}\quad \text { for } j=0, \ldots, m,
 \end{array}
 \right.
\end{gather}
where $\partial \Psi^{j}_{\e}$ denotes the subdifferential of  $\Psi^{j}_{\e}$.
\smallskip

The corresponding explicit Tresca conditions on the interfaces $S_{\e}^j$ with the function $\Psi_{\e}^j$ given in \eqref{Tresca explicit} are as follows:
\begin{gather*}
\left\{
\begin{array}{l}
 |\sigma^{\e}(\nu)_{\tau}| < G_{\e}^j (x) \Rightarrow [({\bf u}_{\e})_{\tau}]_{S_{\e}^j} = 0,\\
 |\sigma^{\e}(\nu)_{\tau}| = G_{\e}^j (x) \Rightarrow \exists \lambda_{\e}^j \in S_{\e}^j \text{ s.t. } |{[({\bf u}_{\e})_\tau]_{S_{\e}^j}}| + \lambda_{\e}^j|\sigma^{\e}(\nu)_\tau| = 0 \text{ a.e. on } S_{\e}^j.
 \end{array}
 \right.
\end{gather*}

Our aim now is to study the behavior of the solutions ${\bf u}_{\e}$ for small values of the parameter $\e$. We will do this by studying the asymptotic behavior of the sequence ${\bf u}_{\e}$ for $\e \rightarrow 0$.
When $\e$ tends to zero, the thin layer $\Omega_{\e}^M$ approaches the interface $\Sigma$. The domain $\Omega_{\e}^{a}$ tends to the domain $\Omega^{a}$.

% \begin{figure}[ht!]
% \centering
%  \includegraphics[width=0.4\textwidth]{tr6.png}
%  \caption{The structure of the domain $\Omega_{\e}$ in the limit $\e = 0$.}
% \end{figure}

\subsection{A priori estimates and existence of solutions for the Problem $\mathcal P_{\e}$}
The first step in the proof of existence of the solution consists in obtaining a bound for minimizing sequences. We use the generic notation $C$ for constants which can be expressed independently of $\e$. 

Since for ${\bf v}=0$ we have ${\cal E}_{\e} ({\bf v})=0$, without lost of generality we can assume that every  field ${\bf u}$ of a minimizing sequence 
satisfies ${\cal E}_{\e} ({\bf u})\leq 0$.
\medskip

Let ${\bf u}$ be in ${\K}_\e$, such that ${\cal E}_{\e} ({\bf u})\leq 0$. Hence we have
\begin{multline} \label{5.0}
{\overline \alpha\over 2}\Big(\e \sum_{j=1}^{m} \|e (u^j)\|^{2}_{L^{2}(\Omega_{\e}^j)} + \e \|e (u)\|^{2}_{L^{2}(\Omega_{\e}^0)} + \|e(u)\|^2_{L^2(\O^{b})} + \|e(u)\|^2_{L^2(\O^{a}_{\e})}\Big) + \sum_{j=0}^{m} \Psi^{j}_{\e}([{\bf u}_{\tau}]_{S^{j}_{\e}})  \\
\leq \int_{\O^*_{\e}} f\cdot u\,dx + \sum_{j=1}^{m} \int_{\Omega_{\e}^j} f^{j}_{\e} \cdot u^j \,dx.
\end{multline}
Now we use \eqref{eqk}$_2$ to get
\begin{equation} \label{f0u0}
\int_{\O^*_{\e}} f\cdot u\;dx \leq C \| f \|_{L^2 (\Omega)} \Big( \| u \|_{L^2(\O^{b})} + \| u \|_{L^2(\O_{\e}^{a})} + \| u\|_{L^2 (\Omega_{\e}^0)} \Big) \leq C \| f \|_{L^2 (\Omega)} {\bf M}({\bf u}).
\end{equation}
The other terms on the right--hand side  for $1\leq j\leq m$ are simply bounded as follows:
\begin{equation} \label{f0uj}
\sum_{j=1}^m\int_{\O_{\e}^{j}}f^{j}_{\e}\cdot u \,dx \leq  \sum_{j=1}^m\|f^{j}_{\e} \|_{L^\infty(\O_{\e}^j)} \; \|u \|_{L^{1}(\O^{j}_{\e})} \leq C \e \max_{j=1,\ldots,m} \|f^{j}_{\e} \|_{L^\infty(\O_{\e}^j)}  \; {\bf M}({\bf u}).
\end{equation}
The last inequality is obtained due to \eqref{eqk}$_5$. Hence
\begin{equation}\label{C0}
\int_{\O^*_{\e}} f\cdot u\,dx + \sum_{j=1}^{m} \int_{\Omega_{\e}^j} f^{j}_{\e} \cdot u^j \,dx\leq C^{'}_0\big(\e  \max_{j=1,\ldots,m} \|f^{j}_{\e} \|_{L^\infty(\O_{\e}^j)}  +  \| f\|_{L^2 (\Omega)} \big){\bf M}({\bf u}).
\end{equation}

\begin{proposition} \label{proposition5.2}(Estimate for minimizing sequences of ${\cal E}_\e$.)
We assume that 
\begin{equation}\label{Cond}
C^{'}_0 \Big(\e  \max_{j=1,\ldots,m} \|f^{j}_{\e} \|_{L^\infty(\O_{\e}^j)}  +  \| f \|_{L^2 (\Omega)} + \sum_{k = 0}^m\Big(\|g^k_{\e}\|_{L^1(S^{k}_\e)}+{a^k_\e\over M^k_\e}\Big)\Big)\max_{k=1,\ldots,m}\Big(\frac{1}{M^k_{\e}}\Big)\leq {1\over 2}.
\end{equation}
Then, there exists a  constant $C$ which does not depend on $\e$, such that  for every field ${\bf u}$ satisfying ${\cal E}_{\e} ({\bf u})\leq 0$ we have
\begin{equation} \label{bound1}
{\bf M}({\bf u}) \leq C \Big(\e  \max_{j=1,\ldots,m} \|f^{j}_{\e} \|_{L^\infty(\O_{\e}^j)}  +  \| f\|_{L^2 (\Omega)} + \sum_{k = 0}^m\Big(\|g^k_{\e}\|_{L^1(S^{k}_\e)}+{a^k_\e\over M^k_\e}\Big)\Big).
\end{equation}
\end{proposition}
\begin{proof} 
From \eqref{5.0} we derive
\begin{multline}
\label{ineq512}
{\overline \alpha\over 2}\Big(\e \sum_{j=1}^{m} \|e (u^j)\|^{2}_{L^{2}(\Omega_{\e}^j)} + \e \|e (u)\|^{2}_{L^{2}(\Omega_{\e}^0)} + \|e(u)\|^2_{L^2(\O^{b})} + \|e(u)\|^2_{L^2(\O^{a}_{\e})}\Big) + \sum_{j=0}^{m} \Psi^{j}_{\e}([{\bf u}_{\tau}]_{S^{j}_{\e}}) \\
\leq C^{'}_0 {\bf M}({\bf u}) \; \Big(\e  \max_{j=1,\ldots,m} \|f^{j}_{\e} \|_{L^\infty(\O_{\e}^j)} + \| f\|_{L^2 (\Omega)} \Big).
\end{multline}
Denote
$$\begin{aligned}
  &\xi_{\bf v} =  \sqrt{\|e(v)\|^2_{L^2(\O^{b} )}+\|e(v)\|^2_{L^2(\O^{a}_\e )}+\e\|e(v)\|^2_{L^2(\O_{\e}^0)}+\sum_{j=1}^m\e\|e(v^{j})\|^2_{L^2(\O_{\e}^{j} )}},\\
  &\eta_{\bf v} = \sum_{j=1}^m \big(\|[{\bf v}_\nu]^{+}_{S^{j}_{\e}}\|_{L^1(S^{j}_\e)} + \|[{\bf v}_\tau]_{S^{j}_{\e}}\|_{L^1(S^{j}_\e)}\big).
 \end{aligned}$$
Remark that
%$${\bf E}({\bf u}) = \xi^2_{\bf u} + \eta_{\bf u}, \qquad {\bf M}({\bf u}) = \xi_{\bf u} + \eta_{\bf u}.$$
$${\bf M}({\bf u}) = \xi_{\bf u} + \eta_{\bf u}.$$
From \eqref{ineq512} we get (recall that $\Psi^{0}_{\e}([{\bf u}_{\tau}]_{S^{j}_{\e}}) $ is non negative)
\begin{equation}
\label{ineq514}
{\overline{\alpha}\over 2}\xi^2_{\bf u} + \sum_{j=1}^{m} \Psi^{j}_{\e}([{\bf u}_{\tau}]_{S^{j}_{\e}}) \leq C^{'}_0 (\xi_{\bf u} + \eta_{\bf u}) \; \Big(\e  \max_{j=1,\ldots,m} \|f^{j}_{\e} \|_{L^\infty(\O_{\e}^j)}  +  \| f \|_{L^2 (\Omega)} \Big).
\end{equation}

\noindent \textit{Step 1.} In this step we obtain a first estimate on  $\eta_{\bf u}$. By definition of ${\K}_\e$ (see \eqref{kaepsilon}) we have
\begin{equation}
\label{ineq-nu}
 \|[{\bf u}_{\nu}]^{+}_{S^{j}_{\e}}\|_{L^1(S^{j}_\e)} \leq \|g^j_{\e}\|_{L^1(S^{j}_\e)}
\end{equation}
and \eqref{Mj}  gives
\begin{equation}
\label{ineq-tau}
\|[{\bf u}_{\tau}]_{S^{j}_{\e}}\|_{L^1(S^{j}_\e)}\leq \frac{1}{M^j_{\e}} \,\Psi^j_{\e}([{\bf u}_{\tau}]_{S^{j}_{\e}})+{a^j_\e\over M^j_\e}.
\end{equation}
Hence,
\begin{equation}
\label{Eq5.18}
\begin{aligned}
\eta_{\bf u}&\leq \sum_{j=1}^m \|g^j_{\e}\|_{L^1(S^{j}_\e)}+\sum_{j=1}^m\Big(\frac{1}{M^j_{\e}} \,\Psi^j_{\e}([{\bf u}_{\tau}]_{S^{j}_{\e}})+{a^j_\e\over M^j_\e}\Big)\\
&\leq \sum_{j=1}^m \|g^j_{\e}\|_{L^1(S^{j}_\e)}+\max_{k=1,\ldots,m}\Big(\frac{1}{M^k_{\e}}\Big)\sum_{j=1}^m \,\Psi^j_{\e}([{\bf u}_{\tau}]_{S^{j}_{\e}})+\sum_{j=1}^m{a^j_\e\over M^j_\e}.
\end{aligned}
\end{equation}
The above inequalities and \eqref{ineq514} lead to
\begin{equation}
\label{ineq_eta}
\begin{aligned}
{\overline{\alpha}\over 2}\xi^2_{\bf u} + \sum_{j=1}^{m} \Psi^{j}_{\e}([{\bf u}_{\tau}]_{S^{j}_{\e}}) \leq &C^{'}_0\Big( \xi_{\bf u} +\sum_{j=1}^m \|g^j_{\e}\|_{L^1(S^{j}_\e)}+\sum_{j=1}^m{a^j_\e\over M^j_\e}\Big)\Big(\e  \max_{j=1,\ldots,m} \|f^{j}_{\e} \|_{L^\infty(\O_{\e}^j)}  +  \| f \|_{L^2 (\Omega)} \Big)\\
+&C^{'}_0 \Big(\e  \max_{j=1,\ldots,m} \|f^{j}_{\e} \|_{L^\infty(\O_{\e}^j)}  +  \| f \|_{L^2 (\Omega)} \Big)\max_{k=1,\ldots,m}\Big(\frac{1}{M^k_{\e}}\Big)\sum_{j=1}^m \,\Psi^j_{\e}([{\bf u}_{\tau}]_{S^{j}_{\e}})
\end{aligned}
\end{equation}
\textit{Step 2.} Estimates on $\xi_{\bf u}$ and $\eta_{\bf u}$. Assuming \eqref{Cond}, one has
\begin{equation}\label{Cond2}
C^{'}_0 \Big(\e  \max_{j=1,\ldots,m} \|f^{j}_{\e} \|_{L^\infty(\O_{\e}^j)}  +  \| f \|_{L^2 (\Omega)} \Big)\max_{k=1,\ldots,m}\Big(\frac{1}{M^k_{\e}}\Big)\leq {1\over 2}.
\end{equation}
Inserting \eqref{Cond2} in \eqref{ineq_eta}, we obtain
\begin{equation}
{ \overline{\alpha}\over 2}\xi^2_{\bf u} \leq  {\overline{\alpha}\over 2}\xi^2_{\bf u} + {1\over 2} \sum_{j=0}^{m} \Psi^{j}_{\e}([{\bf u}_{\tau}]_{S^{j}_{\e}}) \leq C^{'}_0 \Big( \xi_{\bf u} + \sum_{k = 0}^m\|g^k_{\e}\|_{L^1(S^{k}_\e)}+\sum_{j=1}^m{a^j_\e\over M^j_\e} \Big) \; \Big(\e \max_{j=1,\ldots,m} \|f^{j}_{\e} \|_{L^\infty(\O_{\e}^j)} +  \| f \|_{L^2 (\Omega)} \Big) .
\end{equation}
The  classical inequality $\ds ab\leq   {1\over 2}a^{2} +\frac{1}{2}\, b^{2}$ allows to obtain
\begin{equation}
 \xi_{\bf u} \leq C^{'}_1\Big(\sum_{k = 1}^m\Big(\|g^k_{\e}\|_{L^1(S^{k}_\e)}+{a^k_\e\over M^k_\e}\Big)+\e \max_{j=1,\ldots,m} \|f^{j}_{\e} \|_{L^\infty(\O_{\e}^j)} +  \| f \|_{L^2 (\Omega)} \Big). 
\end{equation}
From the last inequality and \eqref{ineq512} it follows that
\begin{equation}
\label{ineq-xi-psi}
 \sum_{j=0}^{m} \Psi^{j}_{\e}([{\bf u}_{\tau}]_{S^{j}_{\e}}) \leq C^{'}_2\Big(\sum_{k = 1}^m\Big(\|g^k_{\e}\|_{L^1(S^{k}_\e)}+{a^k_\e\over M^k_\e}\Big)+\e \max_{j=1,\ldots,m} \|f^{j}_{\e} \|_{L^\infty(\O_{\e}^j)} +  \| f \|_{L^2 (\Omega)} \Big)^2.
\end{equation} 
Combining \eqref{Eq5.18} and \eqref{ineq-xi-psi} and taking into account \eqref{Cond} we get
\begin{equation}
\label{eta-fin}
 \eta_{\bf u} \leq C^{'}_3\Big(\sum_{k = 1}^m\Big(\|g^k_{\e}\|_{L^1(S^{k}_\e)}+{a^k_\e\over M^k_\e}\Big)+\e \max_{j=1,\ldots,m} \|f^{j}_{\e} \|_{L^\infty(\O_{\e}^j)} +  \| f \|_{L^2 (\Omega)} \Big).
\end{equation}
The constants $C^{'}_1$, $C^{'}_2$ and $C^{'}_3$ depend  on $C^{'}_0$ and $\overline{\alpha}$. 
Finally \eqref{bound1} is proved.
\end{proof}
\begin{proposition} \label{proposition5.3}(Existence of solutions for ${\cal P}_\e$.) 
Under assumption \eqref{Cond} there exists at least one global minimizer for the functional ${\cal E}_\e$.
\end{proposition}
\begin{proof} Let $\{{\bf u}_\eta\}_{\eta>0}$ be a minimizing sequence for ${\cal P}_\e$.
Due to \eqref{bound1} and  \eqref{L2R-Final} there exists a constant $C$ which does not depend on $\eta$ (but which depends on $\e$) such that 
$${\bf u}_\eta=(u_\eta, u_\eta^1 \ldots, u_\eta^{m})\in \V_\e\qquad \|u_\eta\|_{ H^1(\O_{\e}^* ; \R^3 )}+\sum_{j=1}^m\|u_\eta^j\|_{H^{1}(\O^{j}_{\e} ; \R^3 )}\leq C.$$

\begin{remark}
Estimate \eqref{bound1} is not an estimate of a norm. If we  use \eqref{eqk}, we deal with some $L^1$-estimates and with the corresponding weak limits which are measures.  That is why we prefer to consider estimate \eqref{L2R-Final}; in this case the weak limit belongs to $\K_\e$. But this estimate is bad for the asymptotic process. 
\end{remark}

Since ${\cal E}_{\e}$ is bounded from below, convex and weakly lower semicontinuous as a sum of weakly lower semicontinuous functions, there exists at least a minimizer ${\bf u}_\e \in {\K}_{\e}$ for ${\cal E}_{\e}$.
\end{proof}
\begin{remark} Let ${\bf u}_\e=(u_\e, u_\e^1 \ldots, u_\e^{m})\in \K_\e$, ${\bf u}^{'}_\e=(u^{'}_\e, u^{' 1}_\e \ldots, u^{' m}_\e)\in \K_\e$ be two minimizers of ${\cal P}_\e$, both fields satisfy \eqref{wf}. Hence
$$u_\e=u^{'}_\e,\qquad u^{ j}_\e-r^{j}_{\bf u_\e}=u^{' j}_\e-r^{' j}_{\bf u'_\e},\qquad j=1,\ldots,m.$$
\end{remark}

\section{Main result}
In this section, we only consider the case of Tresca friction.

\subsection{Hypotheses on $g_{\e}^j$, $G_{\e}^j$ and $f_{\e}^j$}
\label{hpth}
To pass to the limit in the homogenization process, we need structural assumptions concerning the Tresca data which are more precise than those of Section \ref{S5}.
Hence, we assume that there exist
\begin{enumerate}
 \item $g^j \in L^1 (S^j)$, $j = 0,\ldots, m$, such that
 $$g^{j}_{\e} = g^j \Big(\left\{\frac{\cdot}{\varepsilon}\right\}_Y\Big),$$
and therefore $$\mathcal T_{\e}^{bl, j} (g^j_{\e}) (x', y) = g^j(y) \quad \text{ for a.e. } (x', y) \in \omega \times S^j.$$
% Consequently,
%\begin{equation} 
% \label{saut}
%[\widehat u^{\,0}_{\nu}]_{S^{0}}\leq g^{0} \text { on } S^{0}, \qquad {[\widehat u^{j}_{\nu}]}{_{S^{j}}}\leq g^{j}.
%\end{equation}
 \item  $G^j \in {\cal C}_c(\omega\times S^j)$, $j = 0,\ldots, m$, such that
\begin{gather*}
 \mathcal T_{\e}^{bl, j} (G^j_{\e})  \rightarrow G^j\;\; \hbox{strongly in }L^\infty(\omega \times S^j),\\
G^j(x', y) \geq M^j \quad \text{ for any  }(x',y)\in  \omega \times S^j, \enskip M^j >0.
 \end{gather*}
 \item $F^j \in {\cal C}_c(\omega \times Y^j;\R^3)$, $j = 1,\ldots, m$, such that
 \begin{equation}
f^j_{\e} (x) = \frac{1}{\e} F^j \Big(\e\Big[\frac{x'}{\e}\Big]_{Y'}, \e\Big\{\frac{x}{\e}\Big\}_Y\Big), \quad \text{ for a.e.  } x \in  \Omega^j_\e.
\end{equation}
\end{enumerate}
The Assumption \eqref{Cond} becomes
\begin{equation}\label{Cond-fin}
C^{'}_0 \Big(\max_{j=1,\ldots,m} \big(\|F^{j} \|_{L^\infty(\omega\times Y^j ; \R^3)}\big)  +  \| f \|_{L^2 (\Omega;\R^3)} + \sum_{j = 0}^m\|g^j\|_{L^1(S^{j})}\Big)\max_{j=1,\ldots,m}\Big(\frac{1}{M^j}\Big)\leq {1\over 2}.
\end{equation}

\subsection{The limit problem}
We equip the product space 
$$\H=\ds H^1_\Gamma(\Omega^{b};\R^3)\times H^1(\Omega^{a};\R^3)\times L^2(\omega; H^1_{per} (Y^0 ;\R^3 )) \times \prod_{j=1}^m L^2(\omega; H^1(Y^j;\R^3)) \times  [{\cal M}(\omega;\R^3)]^m \times [{\cal M}(\omega;\R^3)]^m$$ with the norm 
$$
\|{{\bf V}}\|_\H=\|v^{b}\|_{H^1(\O^{b};\R^3)}+\|v^{a}\|_{H^1(\O^{a};\R^3)}+\sum_{j=0}^m\|\widehat{v}^j\|_{L^2(\omega; H^1(Y^j;\R^3))}+\|c\|_{[{\cal M}(\omega;\R^3)]^m}+\|d\|_{[{\cal M}(\omega;\R^3)]^m}
$$
where ${{\bf V}}=(v^{b},v^{a},\widehat{v}^0, \ldots, \widehat{v}^m, c, d)\in \H$. 

We set
\begin{equation*}
\begin{aligned}
\V= \Big\{ {{\bf V}} \in \H\; \; | \;\; &v^{b}(x^{'}, 0)=\widehat{v}^0(x^{'}, y_1,y_2,0),\enskip v^{a}(x^{'}, 0)=\widehat{v}^0(x^{'}, y_1,y_2,1)\quad \hbox{for a.e. } (x^{'},y_1,y_2)\in \omega\times Y^{'},\\
& \widehat{v}^j\;\; \hbox{is orthogonal to the rigid displacements},\; j=1,\ldots, m \Big\}. 
\end{aligned}
\end{equation*}
For any ${\bf V}\in \V$, as in Section \ref{S5}, we define $[{\bf V}]_{S^{j}}= (\widehat{v}^{j}- \widehat{v}^0)_{|S^{j}}, \;\; [{\bf V}_\nu]_{S^{j}}= (\widehat{v}^{j}- \widehat{v}^0)\cdot\nu_{|S^{j}}, \;\; [{\bf V}_\tau]_{S^{j}}= [{\bf V}]_{S^{j}}- [{\bf V}_\nu]_{S^{j}}$, $j=1,\ldots m$,  and
$$
\begin{aligned}
& [{\bf V}]_{S^{0}}(x^{'}, y )=\lim_{s\to 0, s>0}\Big( \widehat{v}^0\big(x^{'}, y+s\nu(y)\big)-\widehat{v}^0\big(x^{'}, y-s\nu(y)\big)\Big),\\
& [{\bf V}_\nu]_{S^{0}}(x^{'}, y )=\lim_{s\to 0, s>0}\Big( \widehat{v}^0\big(x^{'}, y+s\nu(y)\big)-\widehat{v}^0\big(x^{'}, y-s\nu(y)\big)\Big)\cdot\nu(y),\\
&[{\bf V}_\tau]_{S^{0}}(x^{'}, y )=[{\bf V}]_{S^{0}}(x^{'}, y )-[{\bf V}_\tau]_{S^{0}}(x^{'}, y )\qquad \hbox{for a.e. } (x^{'},y)\in \omega\times S^{0}.
\end{aligned}
$$
Now we introdce the closed convex set $\K$
\begin{equation*}
\begin{aligned}
\K= \Big\{ {{\bf V}} \in \V\; \; | \;\; &[{\bf V}_{\nu}]_{S^0} \leq g^0\;\; \hbox{in}\;\;  \omega\times S^0,\quad
 [{\bf V}_{\nu}]_{S^j} +s^j\cdot \nu\leq g^j  \;\;\hbox{in }\;\;  \omega\times S^j \hskip-3mm\phantom{\int}^{(}\footnotemark^{)},\\
 &\hbox{where} \;\; \forall y\in Y^{j}\;\; s^{j}(\cdot, y)=c^{j}+d^{j}\land (y-O^{j})\;\; \hbox{in } {\cal M}(\omega\, ;\,\R^3)\;\; j=1,\ldots, m \Big\}. 
\end{aligned}
\end{equation*}
\footnotetext{ {\it 
This condition means
$$\forall\phi\in {\cal C}_c(\omega\times S^j),\enskip\hbox{s.t.} \enskip \phi(x^{'},y)\geq 0\quad \hbox{on }\; \omega\times S^j,\qquad <\phi , s^j\cdot \nu>_{ {\cal C}_c(\omega\times S^j) , {\cal M}(\omega\times S^j)}\leq \int_{\omega\times S^j} \big(g^j-[{\bf V}_{\nu}]_{S^j}\big)\,\phi\, dx^{'}d\sigma(y),\quad j=1,\ldots,m.$$}}

\begin{theorem}
\label{homogenized}
 Assume  $f\in L^2(\Omega; \R^3)$ and   $f_{\e}^j$, $g_{\e}^j$, $G_{\e}^j$ satisfy the hypotheses of Subsection \ref{hpth} and also assume \eqref{Cond-fin}. Suppose that the following assumption holds: there exists a tensor $a^M$ such that, as $\e \to 0$, 
 \begin{equation}
  \label{tensor}
  \mathcal T_{\e} (a_{\e}^M) \to a^M \quad \text{ a.e. in } \omega \times Y.
\end{equation}
Let ${\bf u}_{\e}$ be the solution of Problem \eqref{wf}, there exist a subsequence -still denoted $\e$- and   ${\bf U}=(u^{b},u^{a},\widehat{u}^0, \ldots, \widehat{u}^m, a, b) \in \K$ such that ($j=1,\ldots,m$)
\begin{equation*}
\begin{aligned}
u_\e & \rightharpoonup u^{b}\quad \hbox{weakly in } H^1_\Gamma(\O^{b};\R^3),\\
u_\e(\cdot +\e {\bf e}_3) & \rightharpoonup u^{a}\quad \hbox{weakly in } H^1(\O^{a};\R^3),\\
{\cal T}_\e(u_\e) & \rightharpoonup \widehat{u}^0\quad \hbox{weakly in } L^2(\omega; H^1(Y^0;\R^3)),\\
\e {\cal T}_{\e}(e (u_{\e}) ) &\rightharpoonup e_{y} (\widehat{u}^0) \quad \hbox{weakly  in } L^2(\omega\times Y^0 ; \R^{3\times 3}),\\
{\cal T}_\e(u^{j}_\e-r^{j}_{\bf u_\e}) & \rightharpoonup \widehat{u}^{j}\quad \hbox{weakly in } L^2(\omega; H^1(Y^{j};\R^3)),\\
{\cal T}_\e(u^{j}_\e) & \rightharpoonup \widehat{u}^{j}+r^{j}\quad \hbox{weakly-* in } {\cal M}(\omega\times Y^{j};\R^3),\\
\e {\cal T}_{\e}(e(u^{j}_{\e})) &\rightharpoonup e_{y}(\widehat{u}^j)\quad \text {weakly  in  }L^2(\omega \times Y^{j}; \R^{3\times 3}),\\
a^{j}_{\bf u_\e} & \rightharpoonup a^{j} \quad \hbox{weakly-* in}\; \; {\cal M}(\omega ; \R^3),\\
b^{j}_{\bf u_\e} & \rightharpoonup b^{j} \quad \hbox{weakly-* in}\; \; {\cal M}(\omega ; \R^3)
\end{aligned}
\end{equation*}
where
$$\forall y\in Y^j,\qquad  r^{j}(\cdot, y)=a^{j}+b^{j}\land (y-O^{j})\;\; \hbox{in } {\cal M}(\omega\, ;\,\R^3),\quad j=1,\ldots,m.$$ 
The limit field ${\bf U}$  satisfies the following unfolded problem:
\begin{equation}
\begin{aligned}\label{inflim2}
&\int_{\O^{b}} a^{b} e( u^{b}) : e(v^{b}-u^{b})\,dx+\int_{\O^{a}} a^{a} e( u^{a}) : e(v^{a}-u^{a})\,dx +\sum_{j=0}^{m}  \int_{\omega\times Y^j} a^M e_y(\widehat u^j) : e_y(\widehat v^j - \widehat u^j) \,dx^{'} dy  \\
+& \sum_{j=0}^{m}<G^j , |[{\bf V}_{\tau}]_{S^{j}}+s^j_\tau| - |[{\bf U}_{\tau}]_{S^j}+r^j_\tau|>_{{\cal C}_c(\omega\times S^j) , {\cal M}(\omega\times S^j)}  \\
\geq &  \int_{\O^{b}} f \cdot (v^{b}-u^{b}) \,dx + \int_{\O^{a}} f \cdot (v^{a}-u^{a}) \,dx + \sum_{j = 1}^m \int_{\omega \times Y^j} F^j \cdot (\widehat v^j  - \widehat u^j ) \, dx^{'} dy\\
+ & \sum_{j = 1}^m <F^j , s^{j} - r^j>_{{\cal C}_c(\omega\times Y^j ; \R^3) , {\cal M}(\omega\times Y^j ; \R^3)},\qquad \forall {\bf V}\in \K
\end{aligned}\end{equation}
where 
$$s^0=r^0=0,\quad \hbox{and}\qquad \forall y\in Y^j,\quad  s^{j}(\cdot, y)=c^{j}+d^{j}\land (y-O^{j})\;\; \hbox{in } {\cal M}(\omega\, ;\,\R^3)$$ 
and where $s^j_\tau=s^j-(s^j\cdot \nu)\nu$, $r^j_\tau=r^j-(r^j\cdot \nu)\nu$, $j=1,\ldots,m$.
\end{theorem}
\begin{proof}
Based on the Proposition \ref{proposition5.2} and the assumptions of Section \ref{hpth}, the solution ${\bf u}_\e$ of Problem ${\cal P}_{\e}$ is uniformly bounded with respect to $\e$. Hence,  up to a subsequence of $\e$ (still denoted $\e$), the convergences of Theorem \ref{convergence} hold. Therefore, we should show that the limits furnished by the Theorem \ref{convergence} satisfy a homogenized limit problem. The main point now is to pass to the limit in \eqref{wf}.
\medskip

 Let ${{\bf V}}=(v^{b},v^{a},\widehat{v}^0, \ldots, \widehat{v}^m, c, d)$ be in 
$$\K\cap \ds H^1_\Gamma(\Omega^{b};\R^3)\times H^1(\Omega^{a};\R^3)\times {\cal C}^\infty_c(\omega; H^1_{per} (Y^0 ;\R^3 )) \times \prod_{j=1}^m {\cal C}^\infty_c(\omega; H^1(Y^j;\R^3)) \times  [{\cal C}^\infty_c(\omega;\R^3)]^m \times [{\cal C}^\infty_c(\omega;\R^3)]^m.$$ We  use the following test function ${\bf v}_\e\in {\K}_\e$:

\begin{equation}\label{Eq6.5}
{\bf v}_\e(x)=\left\{
\begin{aligned}
v_\e(x)=&v^{b}(x)\qquad &\hbox{in}\;\;  \O^{b},\\
v_\e(x)=&v^{a}(x-\e {\bf e}_3)\qquad &\hbox{in}\;\;  \O^{a}_\e,\\
v_\e(x)=&\widehat{v}^0\Big(\Big\{{x^{'}\over \e}\Big\},{x_3\over \e}\Big)\qquad &\hbox{in}\;\;  \O^0_\e,\\
v^{j}_\e(x)=&\widehat{v}^j\Big(\Big\{{x^{'}\over \e}\Big\},{x_3\over \e}\Big)+c^j(x^{'})+d^j(x^{'})\land \Big(\Big\{{x^{'}\over \e}\Big\}+{x_3\over \e}{\bf e}_3-O^j\Big)\qquad &\hbox{in}\;\;  \O^j_\e,\\
=&\widehat{v}^j\Big(\Big\{{x^{'}\over \e}\Big\},{x_3\over \e}\Big)+s^j_\e(x)\qquad &\hbox{in}\;\;  \O^j_\e.
\end{aligned}
\right.
\end{equation}
By construction we have
\begin{equation}\label{Eq6.7}
\begin{aligned}
&v_\e  = v^{b}\qquad  \hbox{in}\;\;  \O^{b},\hskip 12mm v_\e(\cdot +\e {\bf e}_3) = v^{a}\qquad \hbox{in}\;\;  \O^{a},\\
&{\cal T}_\e(v_\e) = \widehat{v}^0,\qquad  \e {\cal T}_{\e}(e (v_{\e}) ) = e_{y} (\widehat{v}^0)\qquad \hbox{in}\;\;  \omega \times Y^0,\\
&{\cal T}_\e(v^{j}_\e-s^{j}_\e)  =\widehat{v}^{j}, \qquad   \e {\cal T}_{\e}(e(v^{j}_{\e}))  =e_{y}(\widehat{v}^j)\qquad \hbox{in}\;\;  \omega \times Y^j,\\
&{\cal T}_\e(v^{j}_\e) \longrightarrow  \widehat{v}^{j}+s^{j}\quad \hbox{strongly in } \;\; L^2(\omega ; H^1( Y^j;\R^3)).
\end{aligned}
\end{equation}

We rewrite \eqref{wf} in the form
\begin{multline*}
 \e \sum_{j=1}^{m}\int_{\O_{\e}^j}a_{\e}^M e(u_{\e}^j) : e(v_\e^j)\,dx + \e\int_{\O_{\e}^0} a_{\e}^M e(u_{\e}) : e(v_\e)\,dx + \int_{\O^{b} } a^{b} e( u_{\e}) : e(v_\e)\,dx + \int_{\O^{a}} a^{a}(\cdot+\e{\bf  e}_3)e\big( u_{\e}(\cdot+\e {\bf e}_3)\big) : e(v^{a})\,dx \\
+ \sum_{j=0}^{m} \int_{S_{\e}^j} G_{\e}^j(x) |[({\bf v}_{\e})_\tau]_{S^{j}_{\e}}| d\sigma(x) - \int_{\O_{\e}} {\bf f}_{\e} \cdot {\bf v_\e} \,dx  \geq \e \sum_{j=1}^{m} \int_{\O_{\e}^j} a_{\e}^M e( u_{\e}^j) : e(u_{\e}^j)\,dx + \e\int_{\O_{\e}^0} a_{\e}^M e(u_{\e}) : e(u_{\e})\,dx \\
+ \int_{\O^{b} } a^{b} e( u_{\e}) : e(u_\e)\,dx + \int_{\O^{a}} a^{a} (\cdot+{\bf e}_3)e\big( u_{\e}(\cdot+\e {\bf e}_3) \big): e\big( u_{\e}(\cdot+\e {\bf e}_3)\big)\,dx + \sum_{j=0}^{m} \int_{S_{\e}^j} G_{\e}^j (x) |[({\bf u}_{\e})_{\tau}]_{S^{j}_{\e}}| d\sigma(x) - \int_{\O_{\e}} {\bf f}_{\e} \cdot {\bf u}_{\e} \,dx.
\end{multline*}
% Note, that
% \begin{equation*}
%  e({\bf v}) = \left\{
%  \begin{array}{ll}
%   e(\varphi) C_{\bf \psi}^{a}, & \e \leq x_3,\\
%  \ds e(\varphi) {\bf \psi}\Big(\Big\{\frac{x'}{\e}\Big\}_{Y'}, \frac{x_3}{\e} \Big) + \frac{1}{2\e} \big(\varphi \nabla \psi + (\varphi \nabla \psi)^T \big) , & 0< x_3 < \e,\\
%  e(\varphi) C^{b}_{\bf \psi}, & x_3 \leq 0,
%  \end{array}
% \right.
% \end{equation*}
% where $C_{\psi}^{b}, C_{\psi}^{a}$ are the values of the function $\psi$ for $x_3 < 0$, $x_3 > 1$ respectively.
%\medskip
Therefore, by unfolding and due to the convergences in Theorem \ref{convergence} and the equalities and convergence  in \eqref{Eq6.7}, we have 
\begin{equation*}
\begin{aligned}
 &\lim_{\e \rightarrow 0} \Big(\e \sum_{j=1}^{m} \int_{\Omega_{\e}^j} a_{\e}^M e(u_{\e}^j): e(v^j_\e)\,dx +\e \int_{\Omega_{\e}^0} a_{\e}^M e(u_{\e}): e(v_\e)\,dx\Big)\\
  = &\lim_{\e \rightarrow 0} \Big(\e^2 \sum_{j=1}^{m} \int_{\omega \times Y^j} \mathcal T_{\e} (a^M_{\e}) \mathcal T_{\e} (e(u_{\e}^j)): \mathcal T_{\e} (e(v^j_\e))\,dx' dy+ \e^2  \int_{\omega \times Y^j} \mathcal T_{\e} (a^M_{\e}) \mathcal T_{\e} (e(u_{\e})): \mathcal T_{\e} (e(v_\e))\,dx' dy \Big)\\
=& \sum_{j=0}^{m} \int_{\omega \times Y^j} a^M e_y(\widehat u^j) : e_y (\widehat{v}^j) \, dx' \, dy,
\end{aligned}\end{equation*}
$$\sum_{j = 0}^m \int_{S_{\e}^j} G_{\e}^j  |({\bf v}_{\e})_\tau]_{S^{j}_{\e}}| d\sigma(x) = \sum_{j = 0}^m \int_{\omega \times S^j} \mathcal T_{\e}^{bl, j}(G_{\e}^j ) \mathcal T_{\e}^{bl, j}(|({\bf v}_{\e})_\tau]_{S^{j}_{\e}}|)dx^{'} d\sigma(y) = \sum_{j = 0}^m \int_{\omega \times S^j} G^j |[{\bf V}_{\tau}]_{S^j}+s^j_\tau|  dx'\,  d\sigma(y)$$ where $s^j_\tau=s^j-(s^j\cdot \nu)\nu$.
\medskip

By the lower semi-continuity with respect to weak (or weak-*) convergences, we first obtain
\begin{multline*}
 \liminf_{ \e\to 0} \sum_{j=0}^{m} \int_{\Omega_{\e}^j} \e a_{\e}^M e( u_{\e}^j): e(u_{\e}^j)\,dx = \liminf_{\e\to 0} \e^2 \sum_{j=0}^{m} \int_{\omega \times Y^j} \mathcal T_{\e} (a^M_{\e}) \mathcal T_{\e} (e( u_{\e}^j)): \mathcal T_{\e}(e(u_{\e}^j))\,dx' dy  \\
\geq \sum_{j=0}^{m} \int_{\omega \times Y^j} a^M e_y(\widehat u^j) : e_y(\widehat u^j) \,dx' dy,
\end{multline*}
then, unfolding the term corresponding to the Tresca friction,  passing to the limit and making use of lower semi-continuity gives
\begin{equation*}
  \liminf_{\e\to 0} \sum_{j = 0}^m \int_{S_{\e}^j} G_{\e}^j  |[({\bf u}_{\e})_{\tau}]_{S^{j}_{\e}}| d\sigma(x) \geq \sum_{j = 0}^m < G^j ,  |[{\bf U}_{\tau}]_{S^{j}}+r^j_\tau| >_{{\cal C}_c(\omega\times S^j) , {\cal M}(\omega\times S^j)}
\end{equation*} where $r^j_\tau=r^j-(r^j\cdot \nu)\nu$.
Considering the terms corresponding to the applied forces we get
\begin{gather*}
  \ds \lim_{\e \to 0}\int_{\O^{b} \cup \O^{a}_{\e}} f \cdot u_{\e} \,dx = \lim_{\e \to 0}\Big(\int_{\O^{b}} f \cdot u_{\e} \,dx + \int_{\O^{a}} f (\cdot + \e e_3) \cdot u_{\e} (\cdot + \e e_3) \,dx \Big)\\
  = \int_{\O^{b}} f \cdot u^{b} \,dx+\int_{\O^{a}} f \cdot u^{a} \,dx,\\
  \ds \lim_{\e \to 0} \int_{\Omega_{\e}^0} f \cdot u_{\e} \, dx = \lim_{\e \to 0}\e\int_{\omega \times Y^0} \mathcal T_{\e} (f)\cdot\mathcal T_{\e}(u_{\e})\, dx' \, dy = 0,\\
  \ds \lim_{\e \to 0} \sum_{j = 1}^m\int_{\O_{\e}^j} f_{\e}^j \cdot u_{\e}^j \, dx =  \lim_{\e\to 0} \e \sum_{j = 1}^m \int_{\omega \times Y^j} \mathcal T_{\e}(f_{\e}^j) \cdot \mathcal T_{\e} (u_{\e}^j) dx' \, dy \\
  =\sum_{j = 1}^m \int_{\omega \times Y^j} F^j \cdot \widehat u^j  \, dx' dy+ \sum_{j = 1}^m < F^j , r^j>_{{\cal C}_c({\omega\times Y^j};\R^3),{\cal M}(\omega\times Y^j;\R^3)}.
\end{gather*}
%By \eqref{tough inequality1} if $\ds \frac{1}{\sqrt \e M_{\e}^j} \| f_{\e}^j \|_{L^2 (\Omega_{\e}^j)} \rightarrow 0$, then $\ds \e \int_{\Omega_{\e}^j} f_{\e} \, u_{\e} \, dx \rightarrow 0$.
%If condition \eqref{condition2} is satisfied, then the same convergence holds true.
In a similar way
\begin{gather*}
  \ds \lim_{\e \to 0}\int_{\O^*_{\e}} f\cdot v_\e \,dx = \int_{\O^{b}} f \cdot v^{b} \,dx+\int_{\O^{a}} f \cdot v^{a} \,dx,\\
  \ds \lim_{\e \to 0} \sum_{j = 1}^m \int_{\O_{\e}^j} f_{\e}^j\cdot  v^j_\e \, dx = \sum_{j = 1}^m \int_{\omega \times Y^j} F^j \cdot (\widehat v^j + s^{j}) \, dx' dy.
\end{gather*}
Using the established convergences we obtain
\begin{equation*}
\begin{aligned}
&\int_{\O^{b}} a^{b} e( u^{b}) : e(v^{b})\,dx+\int_{\O^{a}} a^{a} e( u^{a}) : e(v^{a})\,dx+ \sum_{j=0}^{m} \int_{\omega \times Y^j} a^M e_y(\widehat u^j) : e_y (\widehat v^j) \, dx' \, dy\\
  + &\sum_{j = 0}^m  < G^j ,  |[{\bf V}_{\tau}]_{S^{j}}+s^j_\tau| >_{{\cal C}_c(\omega\times S^j) , {\cal M}(\omega\times S^j)}- \int_{\O^{b}} f \cdot v^{b} \,dx - \int_{\O^{a}} f \cdot v^{a} \,dx \\
  - &\sum_{j = 1}^m < F^j , \widehat v^j + s^{j}>_{{\cal C}_c(\omega\times S^j ; \R^3) , {\cal M}(\omega\times S^j ; \R^3)}\\
\geq&\int_{\O^{b}} a^{b} e( u^{b}) : e(u^{b})\,dx+\int_{\O^{a}} a^{a} e( u^{a}) : e(u^{a})\,dx+ \sum_{j=0}^{m} \int_{\omega \times Y^j} a^M e_y(\widehat u^j) : e_y(\widehat u^j) \,dx' dy \\
+& \sum_{j = 0}^m < G^j ,  |[{\bf U}_{\tau}]_{S^{j}}+r^j_\tau| >_{{\cal C}_c(\omega\times S^j) , {\cal M}(\omega\times S^j)}- \int_{\O^{b}}f \cdot u^{b} \,dx -\int_{\O^{a}}f \cdot u^{a} \,dx\\
 -& \sum_{j = 1}^m <F^j , \widehat u^j + r^j>_{{\cal C}_c(\omega\times Y^j; \R^3) , {\cal M}(\omega\times Y^j ; \R^3)}.
\end{aligned}\end{equation*}
Subtracting the terms on the right-hand side from the left-hand side we get \eqref{inflim2}. By a density argument, \eqref{inflim2} holds for any ${\bf V}\in \K$.
\end{proof}

\begin{remark}\label{rem61} Since the functional 
\begin{equation} \label{problem-lim}
\begin{aligned}
{\cal E} ({\bf V})\doteq  \frac{1}{2} \Big(&\int_{\O^{b}} a^{b} e( v^{b}) : e(v^{b})\,dx+\int_{\O^{a}} a^{a} e( v^{a}) : e(v^{a})\,dx+ \sum_{j=0}^{m} \int_{\omega \times Y^j} a^M e_y(\widehat v^j) : e_y (\widehat v^j) \, dx' \, dy\Big)\\
 + &\sum_{j = 0}^m  < G^j ,  |[{\bf V}_{\tau}]_{S^{j}}+s^j_\tau| >_{{\cal C}_c(\omega\times S^j) , {\cal M}(\omega\times S^j)}- \int_{\O^{b}} f \cdot v^{b} \,dx - \int_{\O^{a}} f \cdot v^{a} \,dx \\
  - &\sum_{j = 1}^m < F^j , \widehat v^j + s^{j}>_{{\cal C}_c(\omega\times Y^j ; \R^3) , {\cal M}(\omega\times Y^j ; \R^3)},\quad {{\bf V}}=(v^{b},v^{a},\widehat{v}^0, \ldots, \widehat{v}^m, c, d)\in \K,\\
 & \forall y\in Y^j,\quad  s^{j}(\cdot, y)=c^{j}+d^{j}\land (y-O^{j})\;\; \hbox{in } {\cal M}(\omega\, ;\,\R^3),\quad s^j_\tau=s^j-(s^j\cdot \nu)\nu
\end{aligned}
\end{equation}
is weakly lower semi-continuous and convex over $\K$, Problem \eqref{inflim2} is equivalent to find a minimizer over $\K$ of the functional ${\cal E}$. The field ${\bf U}$ obtained in Theorem \ref{homogenized} is a global minimizer of this functional over $\K$ and every limit point of the sequence $\{{\bf u}_\e\}$ is a global minimizer of ${\cal E}$.
\end{remark}

\begin{lemma} For every $M\geq 0$ there exists a constant $C(M)$ such that for any ${\bf V}\in \K$ satisfying $\ds {\cal E}({\bf V})\leq M$ we have
\begin{equation}\label{EQ67}
\|v^a\|_{H^1(\Omega^a;\R^3)}+\|v^b\|_{H^1(\Omega^b;\R^3)}+\sum_{j=0}^m \|\widehat v^j\|_{L^2(\omega; H^1(Y^j;\R^3))}+\|c\|_{[{\cal M}(\omega;\R^3)]^m }+\|d\|_{[{\cal M}(\omega;\R^3)]^m }\le C(M).
\end{equation}
\end{lemma}

\begin{proof} From the expression \eqref{problem-lim} of ${\cal E}({\bf V})$, we first deduce that  $\ds {\cal E}({\bf V})\leq M$ implies
\begin{equation}\label{EQ68}
\begin{aligned}
{\overline\alpha\over 2}\Big(\|e(v^a)\|^2_{L^2(\Omega^a;\R^6)}+\|e(v^b)\|^2_{L^2(\Omega^b;\R^6)}+\sum_{j=0}^m\|e(\widehat e_y(v^j))\|^2_{L^2(\omega\times Y^j;\R^6)}\Big)\\`
+\sum_{j=0}^m M^j  \|[{\bf V}_{\tau}]_{S^{j}}+s^j_\tau \|_{{\cal M}(\omega\times S^j;\R^3)}-\|f\|_{L^2(\Omega^a;\R^3)}\| v^a\|_{L^2(\Omega^a;\R^3)}-\|f\|_{L^2(\Omega^b;\R^3)}\| v^b\|_{L^2(\Omega^b;\R^3)}\\
-\sum_{j=1}^m \|F^j\|_{L^2(\omega\times Y^j;\R^3)}\|\widehat v^j\|_{L^2(\omega\times Y^j;\R^3)} - \sum_{j = 1}^m \| F^j \|_{{\cal C}_c(\omega\times Y^j ; \R^3) }\| s^{j}\|_{{\cal M}(\omega\times Y^j ; \R^3)}\leq M.
\end{aligned}
\end{equation}
From the Korn inequality and the trace theorem we get
$$
\begin{aligned}
&\|v^a\|_{H^1(\Omega^a;\R^3)}+\|v^b\|_{H^1(\Omega^b;\R^3)}+\sum_{j=0}^m \|\widehat v^j\|_{L^2(\omega; H^1(Y^j;\R^3))}\\
\leq &C\Big(\|e(v^a)\|^2_{L^2(\Omega^a;\R^6)}+\|e(v^b)\|^2_{L^2(\Omega^b;\R^6)}+\sum_{j=0}^m\|e(\widehat e_y(v^j))\|^2_{L^2(\omega\times Y^j;\R^6)}\Big)
\end{aligned}
$$ where constant $C$ depends on the geometry of the sets $\Omega^a$, $\Omega^b$, $\omega$, $S^0$ and $Y^j$ ($j=1,\ldots, m$). 
\smallskip

Then, the above inequality and  \eqref{EQ68} yield
\begin{equation*}
\begin{aligned}
{\overline\beta}\big(\|v^a\|_{H^1(\Omega^a;\R^3)}+\|v^b\|_{H^1(\Omega^b;\R^3)}+\sum_{j=0}^m \|\widehat v^j\|_{L^2(\omega; H^1(Y^j;\R^3))}\big)+\sum_{j=0}^m M^j  \|s^j_\tau \|_{{\cal M}(\omega\times S^j;\R^3)}\\
- \sum_{j = 1}^m \| F^j \|_{{\cal C}_c(\omega\times Y^j ; \R^3) }\| s^{j}\|_{{\cal M}(\omega\times Y^j ; \R^3)}\leq M
\end{aligned}\end{equation*}
where $\overline \beta>0$.
\smallskip

\noindent Now, in the case  $c\in [L^1(\omega;\R^3)]^m$ and $d\in [L^1(\omega;\R^3)]^m$, from Lemma 4.4 of \cite{cior3}, there exist constants $C^{'}$ and $C^{''}$ which depend on $Y^j$ ($j=1,\ldots, m$) such that
$$\|c\|_{[L^1(\omega;\R^3)]^m}+\|d\|_{[L^1(\omega;\R^3)]^m}\leq C^{'}\sum_{j=1}^m  \|s^j\|_{L^1(\omega\times S^j;\R^3)}\leq C^{''}\sum_{j=1}^m \big( \|s^j_\tau \|_{L^1(\omega\times S^j;\R^3)}+\|(s^j_\nu)^+ \|_{L^1(\omega\times S^j)}\big).$$ Due to the definition of $\K$ we have $\|(s^j_\nu)^+ \|_{L^1(\omega\times S^j)}\le \|g^j\|_{L^1(S^j)}+\|[V_\nu]\|_{L^1(\omega\times S^j)}$. Proceeding as in the proofs of Propositions \ref{smallinclusions} and \ref{proposition5.2} and thanks to the condition \eqref{Cond-fin} and the above inequalities, we deduce the existence of a constant $C(M)$ such that  (we recall that $\|s^j \|_{{\cal M}(\omega\times S^j;\R^3)}=\|s^j\|_{L^1(\omega\times S^j;\R^3)}$)
$$\|v^a\|_{H^1(\Omega^a;\R^3)}+\|v^b\|_{H^1(\Omega^b;\R^3)}+\sum_{j=0}^m \|\widehat v^j\|_{L^2(\omega; H^1(Y^j;\R^3))}+\|c\|_{[L^1(\omega;\R^3)]^m }+\|d\|_{[L^1(\omega;\R^3)]^m }\le C(M).$$
Then, for  general elements ${\bf V}\in \K$, using regularization by convolution in $x'$ and this last estimate we obtain \eqref{EQ67}.
\end{proof}
Independently of Remark \ref{rem61}, using the above lemma we can easily prove the existence of at least one global minimizer for the functional ${\cal E}$.

\begin{remark} Let ${\bf U}$ and ${\bf U}^{'}$ be two global minimizers of the functional ${\cal E}$, from \eqref{inflim2}, we deduce that
$$u^a=u^{'a},\quad u^b=u^{'b},\quad \widehat u^j=\widehat u^{'j},\quad j=0,\ldots, m$$ and 
$$
\begin{aligned}
&\sum_{j=0}^{m}<G^j , |[{\bf U}_{\tau}]_{S^{j}}+r^j_\tau| >_{{\cal C}_c(\omega\times S^j) , {\cal M}(\omega\times S^j)}-\sum_{j = 1}^m <F^j , r^{j} >_{{\cal C}_c(\omega\times Y^j ; \R^3) , {\cal M}(\omega\times Y^j ; \R^3)}\\
=&
\sum_{j=0}^{m}<G^j , |[{\bf U}_{\tau}]_{S^j}+r^{'j}_\tau|>_{{\cal C}_c(\omega\times S^j) , {\cal M}(\omega\times S^j)}-\sum_{j = 1}^m <F^j ,  r^{'j}>_{{\cal C}_c(\omega\times Y^j ; \R^3) , {\cal M}(\omega\times Y^j ; \R^3)}.
\end{aligned}
$$
\end{remark}

\begin{lemma} Let ${\bf u}_\e$ be a minimizer of ${\cal P}_\e$ over $\K_\e$, 
$$m_\e=\min_{{\bf v}\in \K_\e}{\cal E}_\e({\bf v})={\cal E_{\e}}({\bf u}_\e).$$ The whole sequence $\{m_\e\}$ converges and we have
$$m=\min_{{\bf V}\in \K}{\cal E}({\bf V})={\cal E}({\bf U})=\lim_{\e\to 0} m_\e=\lim_{\e\to 0} {\cal E}_\e({\bf u}_\e)\leq 0.$$
\end{lemma}

\begin{proof} As a consequence of Theorem \ref{homogenized} we have (with the subsequence of $\e$ introduced in this theorem)
$$m={\cal E}({\bf U})\leq \liminf_{\e\to 0} m_\e=\liminf_{\e\to 0} {\cal E}_\e({\bf u}_\e).$$
Now, let ${\bf V}$ be in $\K$ and let $\{{\bf V}^\eta\}_\eta$ be a sequence of fields in 
$$\K\cap \ds H^1_\Gamma(\Omega^{b};\R^3)\times H^1(\Omega^{a};\R^3)\times {\cal C}^\infty_c(\omega; H^1_{per} (Y^0 ;\R^3 )) \times \prod_{j=1}^m {\cal C}^\infty_c(\omega; H^1(Y^j;\R^3)) \times  [{\cal C}^\infty_c(\omega;\R^3)]^m \times [{\cal C}^\infty_c(\omega;\R^3)]^m$$ strongly converging to ${\bf V}$ in $\V$ (that is possible using regularization by convolution in $x'$). We fix $\eta$ and  for every ${\bf V}^\eta$  we consider the test field belonging to $\K_\e$ defined by \eqref{Eq6.5} and here denoted ${\bf v^\eta_\e}$. Using the subsequence of $\e$ introduced in Theorem \ref{homogenized},  we can  pass to the limit ($\e\to 0$). Since $m_\e\leq {\cal E_{\e}}({\bf v^\eta_\e})$,  due to \eqref{Eq6.7} we obtain
$$\limsup_{\e \to 0}m_\e\leq \limsup_{\e \to 0}{\cal E_{\e}}({\bf v^\eta_\e})= \lim_{\e \to 0}{\cal E_{\e}}({\bf v^\eta_\e})={\cal E}({\bf V}^\eta).$$
Now $\eta$ goes to $0$; that gives
$$\forall {\bf V}\in \K,\qquad \limsup_{\e \to 0}m_\e \leq {\cal E}({\bf V}).$$
Then
$$ \limsup_{\e \to 0}m_\e \leq {\cal E}({\bf U})\leq \min_{{\bf V}\in \K}{\cal E}({\bf V}).$$
Finally we get $\ds m=\lim_{\e\to 0} m_\e$ and this result holds for the whole sequence $\{\e\}$.
\end{proof}

\subsection{Computation of the effective outer--plane properties for a case of heterogeneous layer without contact}
Solving \eqref{inflim2} numerically is difficult because of the presence of non-differentiable non-linear term $|[{\bf U}_{\tau}]_{S^{j}}|$. Also a complete scale separation is impossible for a non--linear problem. 

In this subsection due to application purposes we want to consider a case without contact. Thus, we assume that the open sets $Y^j$, $j=1,\ldots, m$, are holes. For this case we can give a linear problem whose solution is the couple $(u^a, u^b)$ with Robin-type condition on the interface. The result obtained will be similar to \cite{cior, nrj, geym}.
\smallskip

The space $\V$ is replaced by
\begin{equation*}
\begin{aligned}
\V^{'}= \Big\{ &{{\bf V}}=(v^b, v^a, \widehat{v}^0) \in \H^{'}\; \; | \;\; v^{b}(x^{'}, 0)=\widehat{v}^0(x^{'}, y_1,y_2,0),\enskip v^{a}(x^{'}, 0)=\widehat{v}^0(x^{'}, y_1,y_2,1)\quad \hbox{for a.e. } (x^{'},y_1,y_2)\in \omega\times Y^{'} \Big\}
\end{aligned}
\end{equation*}
where $\H^{'}=\ds H^1_\Gamma(\Omega^{b};\R^3)\times H^1(\Omega^{a};\R^3)\times L^2(\omega; H^1_{per} (Y^0 ;\R^3 ))$. We endowed this Hilbert space with the norm 
$$
\|{{\bf V}}\|_{\V^{'}}=\|v^{b}\|_{H^1(\O^{b};\R^3)}+\|v^{a}\|_{H^1(\O^{a};\R^3)}+\|\widehat{v}^0\|_{L^2(\omega; H^1(Y^0;\R^3))}.
$$
Now, the closed convex set $\K$ is replaced by the space  $\V^{'}$. Since for any ${\bf V}\in \V^{'}$, $-{\bf V}$ also belongs to $\V^{'}$,  the problem \eqref{inflim2} becomes
\begin{equation}
\begin{aligned}\label{reg2}
&\hbox{ Find}\;\; {\bf U}=(u^b, u^a, \widehat{u}^0)\in \V^{'}\;\hbox{s. t. }\\
&\int_{\O^{b}} a^{b} e( u^{b}) : e(v^{b})\,dx+\int_{\O^{a}} a^{a} e( u^{a}) : e(v^{a})\,dx + \int_{\omega\times Y^0} a^M e_y(\widehat u^0) : e_y(\widehat v^0) \,dx^{'} dy  \\
= &  \int_{\O^{b}} f \cdot v^{b} \,dx + \int_{\O^{a}} f \cdot v^{a} \,dx ,\qquad \forall {\bf V}=(v^b, v^a, \widehat{v}^0) \in \V^{'}
\end{aligned}\end{equation}
Introduce the Hilbert space
$$\widehat V_0 =  \Big\{ \widehat{w} \in H^1_{per} (Y ^0;\R^3) \quad | \quad \widehat{w}(y_1, y_2, 0) =\widehat{w}(y_1, y_2, 1) = 0\;\; \hbox{for a.e. } (y_1,y_2)\in Y^{'} \Big\}.$$ 
We consider the 3 corrector displacements $\widehat{\chi}^{i}\in L^\infty(\omega ; H^1_{per}(Y^0;\R^3))$, $i=1,2,3$, defined by 
\begin{equation}
\label{cell-pr}
 \left\{
\begin{aligned}
& \widehat{\chi}^{i}(x^{'},y_1,y_2,1)={\bf e}_i,\quad \widehat{\chi}^{i}(x^{'}y_1,y_2,0)= 0,\qquad \hbox{for a.e. } (x^{'},y_1,y_2)\in \omega\times Y^{'}\\
&\int_{Y^{0}} a^M e_y (\widehat{\chi}^{i}) : e_y (\psi) \, dy = 0 \qquad \hbox{for a.e. } x^{'}\in \omega,\;\; \forall \psi\in \widehat{V}_0,
\end{aligned}
\right.
\end{equation}
The displacements ${\bf e}_i-\widehat{\chi}^{i}\in L^\infty(\omega; H^1_{per}(Y^{0};\R^3))$, $i=1,2,3$, satisfiy
\begin{equation*}
 \left\{
\begin{aligned}
&({\bf e}_i- \widehat{\chi}^{i})(x^{'},y_1,y_2,1)=0,\quad ({\bf e}_i- \widehat{\chi}^{i})(x^{'},y_1,y_2,0)={\bf e}_i,\qquad \hbox{for a.e. } (x^{'},y_1,y_2)\in \omega\times Y^{'}\\
&\int_{Y^0} a^M e_y ({\bf e}_i- \widehat{\chi}^{i}) : e_y (\psi) \, dy = 0 \qquad\hbox{for a.e. } x^{'}\in \omega,\;\;  \forall \psi\in \widehat{V}_0.
\end{aligned}
\right.
\end{equation*}
Below we give the variational problem satisfied by $(u^a, u^b)$.
\begin{theorem} Let ${\bf U}=(u^b, u^a, \widehat u^0)$  be the solution of \eqref{reg2}. We have
\begin{equation}\label{wu0}
\widehat{u}^0(x^{'}, y)=\sum_{i=1}^3 \widehat{\chi}^{i}(x^{'}, y) u^a_{i|\Sigma}(x^{'})+\sum_{i=1}^3 ({\bf e}_i- \widehat{\chi}^{i}(x^{'}, y)) u^b_{i|\Sigma}(x^{'})\qquad \hbox{for a.e. } (x^{'},y)\in \omega\times Y^{0}
\end{equation} and the couple $(u^a, u^b)$ is the solution of the following variational problem:
\begin{equation}\label{inflim2-bis}
 \left\{
\begin{aligned}
&\text{Find } (u^a, u^b)\in H^1(\O^a;\R^3)\times H^1_\Gamma(\O^b;\R^3)\;\;  \text{ s.t. }\\
&\int_{\O^{b}} a^{b} e( u^{b}) : e(v^{b})\,dx+\int_{\O^{a}} a^{a} e( u^{a}) : e(v^{a})\,dx + \int_{\omega} H (u^a_{|\Sigma} - u^b_{|\Sigma}) \cdot (v^a_{|\Sigma} - v^b_{|\Sigma})\,dx^{'} \\
= &\int_{\O^{b}} f \cdot v^{b} \,dx + \int_{\O^{a}} f \cdot v^{a} \,dx,\\
&\forall (v^a, v^b)\in H^1(\O^a;\R^3)\times H^1_\Gamma(\O^b;\R^3),
\end{aligned}
\right.
\end{equation}
where $H$ is the $3\times 3$ symmetric matrix with coefficients in $L^\infty(\omega)$ defined by
\begin{equation}\label{H}
H_{i j}   := \int_{Y^0}  a^M  e_y (\widehat \chi^i) : e_y (\widehat \chi^j) \, dy, \quad i, j = 1,2,3.
\end{equation}
Matrix $H$ is the homogenized tensor of the effective outer-plane stiffness and the $\widehat{\chi}^j$'s ($j=1,2,3$) are the solution of the cell-problem \eqref{cell-pr}. 
\end{theorem}

\begin{proof}
Take $\widehat w\in \widehat V_0$ as test-displacement in  \eqref{reg2}. That gives
$$\forall \widehat w\in \widehat V_0\qquad \int_{\omega\times Y^0} a^M e_y(\widehat u^0) : e_y(\widehat w) \,dx^{'} dy=0.$$
Hence, using the corrector displacements $\widehat \chi^i$, $i=1,2,3$, we express $\widehat u^0$.
\smallskip

Let $(v^a, v^b)$ be in $H^1(\O^a;\R^3)\times H^1_\Gamma(\O^b;\R^3)$, we set $\ds \widehat v^0=\sum_{i=1}^3 \widehat{\chi}^{i} v^a_{i|\Sigma}+\sum_{i=1}^3 ({\bf e}_i- \widehat{\chi}^{i}) v^b_{i|\Sigma}$. We consider ${\bf V}=(v^b, v^a, \widehat v^{0})$ as test function in  \eqref{reg2}. Then we develop $\ds\int_{\omega\times Y^0} a^M e_y(\widehat u^0) : e_y(\widehat v^0) \,dx^{'} dy$ and we obtain \eqref{inflim2-bis}.
\end{proof} 
\medskip

This work was supported by Deutsche Forschungsgemeinschaft (Grants No. OR 190/4-2 and OR 190/6-1).

\end{document}